\documentclass[a4paper,12pt]{amsart}
\pdfoutput=1
\usepackage[utf8]{inputenc}
\usepackage[T1]{fontenc}  
\usepackage{lmodern}
\usepackage{url}
\usepackage{amsmath,amssymb,amsfonts,amsthm}
\usepackage{pstricks,pst-tree,pst-node} 
\usepackage{pst-plot}
\usepackage{array}
\usepackage{enumerate} 
\usepackage{fancyhdr}
\usepackage{pst-circ}
\usepackage{mathrsfs}
\usepackage{float}
\usepackage{graphics}
\usepackage{multirow}
\usepackage{caption}

\renewcommand{\tilde}[1]{\widetilde{#1}}

\renewcommand*{\bar}[1]{\overline{#1}}

\renewcommand*{\ker}{\mathrm{Ker}} 

\newcommand*{\Id}{\mathrm{Id}}

\newcommand*{\vect}[1]{\mathrm{Vect}\{#1\}}

\newcommand*{\rg}{\mathrm{rg}}
\newcommand*{\rk}{\mathrm{rk}}
\newcommand*{\Gal}{\mathrm{Gal}}

\newcommand*{\vcd}{\mathrm{vcd}}
\newcommand*{\cd}{\mathrm{cd}}
\newcommand*{\gd}{\mathrm{gd}}
\newcommand*{\SL}{\mathrm{SL}}
\newcommand*{\GL}{\mathrm{GL}}
\newcommand*{\SO}{\mathrm{SO}}
\newcommand*{\PSO}{\mathrm{PSO}}
\newcommand*{\PSL}{\mathrm{PSL}}
\newcommand*{\PSp}{\mathrm{PSp}}
\newcommand*{\PSU}{\mathrm{PSU}}
\newcommand*{\SU}{\mathrm{SU}}
\newcommand*{\Sp}{\mathrm{Sp}}
\newcommand*{\U}{\mathrm{U}}
\newcommand*{\Ho}{\mathrm{H}}

\newcommand*{\Aut}{\mathrm{Aut}}
\newcommand*{\Inn}{\mathrm{Inn}}
\newcommand*{\Out}{\mathrm{Out}}
\newcommand*{\Ad}{\mathrm{Ad}}
\newcommand*{\Isom}{\mathrm{Isom}}
\newcommand*{\Hom}{\mathrm{Hom}}

\newtheorem{thm}{Theorem}[section]
\newtheorem{cor}{Corollary}[section]
\newtheorem{lemma}{Lemma}[section]
\newtheorem{prop}{Proposition}[section]
\newtheoremstyle{mth}{0.15cm}{0.15cm}{\itshape}{}{\scshape}{}{0.5em}{\textbf{\thmname{#1}}}
\theoremstyle{mth}
\newtheorem{mth}{Main Theorem}[section]
\newtheoremstyle{de}{0.15cm}{0.15cm}{\upshape}{}{\upshape}{}{0.5em}{\underline{\textbf{\thmname{#1} \thmnumber{#2}}} #3}
\theoremstyle{de}

\newtheoremstyle{rmq}{0.15cm}{0.15cm}{}{}{\itshape}{}{0.5em}{\thmname{#1} \thmnumber{#2} #3}
\theoremstyle{rmq}
\newtheorem{rmq}{Remark}[section]
\newtheoremstyle{ex}{0.15cm}{0.15cm}{}{}{\itshape}{}{0.5em}{\thmname{#1} \thmnumber{#2} #3}
\theoremstyle{ex}

\newenvironment{demo}{\begin{proof}}{\end{proof}}

\author{Cyril Lacoste}

\title{Dimension rigidity of lattices in semisimple Lie groups}

\makeatletter

\@addtoreset{equation}{section}
\makeatother

\begin{document}

\maketitle

\begin{abstract}
We prove that if $\Gamma$ is a lattice in the group of isometries of a symmetric space of non-compact type without euclidean factors,  then the virtual cohomological dimension of $\Gamma$ equals its proper geometric dimension.
\end{abstract}

\section{Introduction}

Let $\Gamma$ be a discrete virtually torsion-free group. There exist several notions of "dimension" for $\Gamma$. One of them is the \textit{virtual cohomological dimension} $\vcd(\Gamma)$, which is the cohomological dimension of any torsion-free finite index subgroup of $\Gamma$. Due to a result by Serre, it does not depend on the choice of such a subgroup (see \cite{Brown}). Another one is the \textit{proper geometric dimension}. A $\Gamma$-CW-complex $X$ is said to be a model for $\underline{E}\Gamma$ if the stabilizers of the action of $\Gamma$ on $X$ are finite and for every finite subgroup $H$ of $\Gamma$, the fixed point space $X^H$ is contractible. Note that two models for $\underline{E}\Gamma$ are $\Gamma$-equivariantly homotopy equivalent to each other. The \textit{proper geometric dimension} $\underline{\gd}(\Gamma)$ of $\Gamma$ is the smallest possible dimension of a model for $\underline{E}\Gamma$.

These two notions are related. In fact, we always have the inequality
\[ \vcd(\Gamma) \leqslant \underline{\gd}(\Gamma) \]
but this inequality may be strict, see for instance the construction of Leary and Nucinkis in \cite{Leary-Nucinkis}, or other examples in \cite{Brady},
\cite{Leary}, \cite{Martinez}, \cite{Degrijse-Petrosyan}, \cite{Degrijse-Souto}.

However there are also many examples of virtually torsion-free groups $\Gamma$ with $\vcd(\Gamma)=\underline{\gd}(\Gamma)$. For instance in \cite{Degrijse} Degrijse and Martinez-Perez prove that this is the case for a large class of groups containing all finitely generated Coxeter groups. Other examples for equality can be found in \cite{Aramayona}, \cite{Aramayona-Martinez}, \cite{Luck} and \cite{Vogtmann}.

In this paper we will prove that equality holds for groups acting by isometries, discretely and with finite covolume on symmetric spaces of non-compact type without euclidean factors:

\begin{thm}\label{Théorème principal}
Let $S$ be a symmetric space of non-compact type without euclidean factors. Then
\[ \underline{\gd}(\Gamma)=\vcd(\Gamma) \]
for every lattice $\Gamma \subset \Isom(S)$.
\end{thm}

Recall that a symmetric space of non-compact type without euclidean factors is of the form $G/K$ where $G$ is a semisimple Lie group, which can be assumed to be connected and centerfree, and $K \subset G$ is a maximal compact subgroup. Then $\Isom(S)=\Aut(\mathfrak{g})=\Aut(G)$ where $\mathfrak{g}$ is the Lie algebra of $G$, and note that this group is semisimple, linear and algebraic but may be not connected. In \cite{Aramayona} the authors prove Theorem \ref{Théorème principal} for lattices in classical simple Lie groups $G$. We will heavily rely on their results and techniques.

We discuss now some applications of Theorem \ref{Théorème principal}. First note that the symmetric space $S$ is a model for $\underline{E}\Gamma$. Theorem \ref{Théorème principal} yields then that:

\begin{cor}\label{Corollaire modèle cocompact}
If $S$ is a symmetric space of non-compact type and without euclidean factors, and if $\Gamma \subset \Isom(S)$ is a lattice, then $S$ is $\Gamma$-equivariantly homotopy equivalent to a proper cocompact $\Gamma$-CW complex of dimension $\vcd(\Gamma)$.
\end{cor}

We stress again that in the setting of Theorem \ref{Théorème principal} we are considering the full group of isometries of $S$. This has the consequence that we are able to deduce that there is equality between the virtual cohomological dimension and the proper geometric dimension not only for lattices in $\Isom(S)$, but also for groups abstractly commensurable to them. Here, two groups $\Gamma_1$ and $\Gamma_2$ are said \textit{abstractly commensurable} if for $i=1,2$, there exists a subgroup $\tilde{\Gamma_i}$ of finite index in $\Gamma_i$, such that $\tilde{\Gamma_1}$ is isomorphic to $\tilde{\Gamma_2}$. Then we obtain from Theorem \ref{Théorème principal} that:

\begin{cor}\label{Corollaire commensurable}
If a group $\Gamma$ is abstractly commensurable to a lattice in the group of isometries of a symmetric space of non-compact type without euclidean factors, then $\underline{\gd}(\Gamma)=\vcd(\Gamma)$.
\end{cor}

\begin{rmq}
Note that in general the equality between the proper geometric dimension and the virtual cohomological dimension behaves badly under commensuration. For instance, the fact that there exist virtully torsion-free groups $\Gamma$ with $\vcd(\Gamma) \geqslant 3$ and such that $\vcd(\Gamma) < \underline{\gd}(\Gamma)$ proves that if $\Gamma'$ is a torsion-free subgroup of $\Gamma$ of finite index, then $\vcd(\Gamma')=\cd(\Gamma')=\gd(\Gamma')=\underline{\gd}(\Gamma')$, whereas $\Gamma$ is commensurable to $\Gamma'$ and $\vcd(\Gamma) < \underline{\gd}(\Gamma)$. In fact, we have concrete exemples of groups for which Corollary \ref{Corollaire commensurable} fails among familiar classes of groups. For instance, in \cite{Degrijse} the authors prove that if $\Gamma$ is a finitely generated Coxeter group then $\vcd(\Gamma)=\underline{\gd}(\Gamma)$ and in \cite{Leary} the authors construct finite extensions of certain right-angled Coxeter groups such that $\vcd(\Gamma) < \underline{\gd}(\Gamma)$.
\end{rmq}

Returning to the applications of Theorem \ref{Théorème principal}, we obtain from Corollary \ref{Corollaire commensurable} that lattices in $\Isom(S)$ are dimension rigid in the sense of \cite{Degrijse-Souto}: we say that a virtually torsion-free group $\Gamma$ is \textit{dimension rigid} if one has $\underline{\gd}(\tilde{\Gamma})=\vcd(\tilde{\Gamma})$ for every group $\tilde{\Gamma}$ which contains $\Gamma$ as a finite index normal subgroup.

Dimension rigidity has a strong impact on the behaviour of the proper geometric dimension under group extensions, and we obtain from Corollary \ref{Corollaire commensurable} and \cite[Cor. 2.3]{Degrijse-Petrosyan} that:

\begin{cor}
If $\Gamma$ is a lattice in the group of isometries of a symmetric space of non-compact type without euclidean factors and 
\[ 1 \rightarrow \Gamma \rightarrow G \rightarrow Q \rightarrow 1 \]
is a short exact sequence, then $\underline{\gd}(G) \leqslant \underline{\gd}(\Gamma) + \underline{\gd}(Q)$.
\qed
\end{cor}

We sketch now the strategy of the proof of Theorem \ref{Théorème principal}. To begin with, note that while symmetric spaces, both Riemannian and non-Riemannian, will play a key role in our considerations, most of the time we will be working in the ambient Lie group. In fact it will be convenient to reformulate Theorem \ref{Théorème principal} as follows:

\begin{mth}
Let $\mathfrak{g}$ be a semisimple Lie algebra. Then 
\[ \underline{\gd}(\Gamma)=\vcd(\Gamma) \]
for every lattice $\Gamma \subset \Aut(\mathfrak{g})$.
\end{mth}

The key ingredient in the proof of the Main Theorem, and hence of Theorem \ref{Théorème principal}, is a result of Lück and Meintrup \cite{Luck-Meintrup}, which basically asserts that the proper geometric dimension $\underline{\gd}(\Gamma)$ equals the Bredon cohomological dimension $\underline{\cd}(\Gamma)$ - see Theorem \ref{Théorème Luck-Meintrup} for a precise statement. In the light of this theorem it suffices to prove that the two cohomological notions of dimension $\vcd(\Gamma)$ and $\underline{\cd}(\Gamma)$ coincide. In \cite{Aramayona} the authors noted that to prove the equality $\vcd(\Gamma)=\underline{\cd}(\Gamma)$ it suffices to ensure that the fixed point sets $S^{\alpha}$ of finite order elements $\alpha \in \Gamma$ are of small dimension - see Section 2.8 for details. Still in \cite{Aramayona} the authors checked that this was the case for lattices contained in the classical simple Lie groups. We will use a similar strategy to prove the Main Theorem for lattices in groups of automorphisms of all simple Lie algebras. Recall that any non-compact finite dimensional simple Lie algebra over $\mathbb{R}$ is either isomorphic to one of the classical types or to one of the exceptional ones. The classical Lie algebras are the complex ones
\[ \mathfrak{sl}(n,\mathbb{C}), \mathfrak{so}(n,\mathbb{C}), \mathfrak{sp}(2n,\mathbb{C}) \]
and their real forms 
\[ \mathfrak{sl}(n,\mathbb{R}), \mathfrak{sl}(n,\mathbb{H}), \mathfrak{so}(p,q), \mathfrak{su}(p,q), \mathfrak{sp}(p,q), \mathfrak{sp}(2n,\mathbb{R}), \mathfrak{so}^*(2n). \]
Similarly, the exceptional Lie algebras are the five complex ones
\[ \mathfrak{g}_2^{\mathbb{C}}, \mathfrak{f}_4^{\mathbb{C}}, \mathfrak{e}_6^{\mathbb{C}}, \mathfrak{e}_7^{\mathbb{C}}, \mathfrak{e}_8^{\mathbb{C}} \]
and their twelve real forms
\[ \mathfrak{g}_{2(2)}, \mathfrak{f}_{4(4)}, \mathfrak{f}_{4(-20)}, \mathfrak{e}_{6(6)}, \mathfrak{e}_{6(2)}, \mathfrak{e}_{6(-14)}, \mathfrak{e}_{6(-26)}, \]
\[ \mathfrak{e}_{7(7)}, \mathfrak{e}_{7(-5)}, \mathfrak{e}_{7(-25)}, \mathfrak{e}_{8(8)}, \mathfrak{e}_{8(-24)}. \]
Here the number in brackets is the difference between the dimension of the adjoint group and twice the dimension of a maximal compact subgroup (which equals 0 for a complex Lie group).

We illustrate now the basic steps of the proof of the Main Theorem in the example of $\mathfrak{g}=\mathfrak{sl}(n,\mathbb{C})$. Suppose that $\Gamma \subset \Aut(\mathfrak{g})$ is a lattice, and consider the symmetric space $S=\PSL(n,\mathbb{C})/\PSU_n$. To prove that $\underline{\gd}(\Gamma)=\vcd(\Gamma)$, it will suffice to establish that
\begin{equation}
\label{Relation sln}
 \dim S^{\alpha} < \dim S - \rk_{\mathbb{R}}(\PSL(n,\mathbb{C}))=(n^2-1)-(n-1)
\end{equation}
for every $\alpha \in \Gamma$ of finite order and non central (see Lemma \ref{Lemme clé}).  First note that $\alpha$ is the composition of an inner automorphism and an outer automorphism. Since every non-trivial element in $\Out(\mathfrak{g})$ has order 2, it follows that $\alpha^2$ is an inner automorphism. If it is non trivial then we use the results of Section 6.1 in \cite{Aramayona}. In general, if $\alpha=\Ad(A)$ is a non-trivial inner automorphism of $\mathfrak{sl}(n,\mathbb{C})$, we get from \cite{Aramayona} that \eqref{Relation sln} holds.

We are reduced to the case where $\alpha^2$ is trivial, meaning that $\alpha$ is of order 2. Then the automorphism $\alpha \in \Aut(\mathfrak{g})$ is induced by an automorphism of the adjoint group $G_{ad}=\PSL(n,\mathbb{C})$ which is still denoted $\alpha$ and is also an involution. The fixed point set $S^{\alpha}$ is the Riemannian symmetric space associated to $G_{ad}^{\alpha}$, where $G_{ad}^{\alpha}$ is the set of fixed points of $\alpha$. Now, notice that the quotient $G_{ad}/G_{ad}^{\alpha}$ is a (non-Riemannian) symmetric space. The symmetric spaces associated to simple groups have been classified by Berger in \cite{Berger}. In the case of $G_{ad}=\PSL(n,\mathbb{C})$, we obtain from this classification that the Lie algebra of $G_{ad}^{\alpha}$ is either compact or isomorphic to $\mathfrak{so}(n,\mathbb{C})$, $\mathfrak{s(gl(}k,\mathbb{C}) \oplus \mathfrak{gl}(n-k,\mathbb{C}))$, $\mathfrak{sp}(n,\mathbb{C})$, $\mathfrak{sl}(n,\mathbb{R})$,  $\mathfrak{su}(p,n-p)$ or $\mathfrak{sl}(\frac{n}{2},\mathbb{H})$, where $\mathfrak{sp}(n,\mathbb{C})$ and $\mathfrak{sl}(\frac{n}{2},\mathbb{H})$ only appear if $n$ is even. Armed with this information, we check \eqref{Relation sln} for every involution $\alpha$, which leads to the Main Theorem for $\mathfrak{g}=\mathfrak{sl(n,\mathbb{C})}$. The argument we just sketched will be applied in Section 3 to all complex simple Lie algebras and in Section 4 to the real ones. Since the arguments are similar, and since the complex case is somewhat easier, we advise the reader to skip Section 4 in a first reading.

Having dealt with the simple Lie algebras, we treat in section 5 the semisimple case. The method for the simple algebras will not work at first sight, but the proof will eventually by simpler. The idea is to restrict to irreducible lattices, i.e. those who cannot be decomposed into a product. Then we will show that the rational rank of an irreducible lattice is lower than the real rank of any factor of the adjoint group, meaning that we get a much improved bound than in \eqref{Relation sln}. This fact will lead rapidly to the Main Theorem.

\vspace{0.5 cm}

Finally note that in the proof of the Main Theorem, we do not construct a concrete model for $\underline{E}\Gamma$ of dimension $\vcd(\Gamma)$, we just prove its existence. It is however worth mentioning that in a few cases such models are known. For instance if $\Gamma=\SL(n,\mathbb{Z})$, the symmetric space $S=\SL(n,\mathbb{R})/\SO_n$ admits a $\Gamma$-equivariant deformation retract of dimension $\vcd(\Gamma)$ called the "well-rounded retract" (see \cite{Ash}, \cite{Pettet-Souto} and \cite{Souto-Pettet}). It will be interesting to do the same for groups such as $\Sp(2n,\mathbb{Z})$.

\vspace{0.5cm}

\noindent \textbf{Acknowledgements.} The author thanks Dave Witte Morris for his help, Dieter Degrijse for interesting discussions and Juan Souto for his useful advice and instructive discussions.

\section{Preliminaries}

In this section we recall some basic facts and definitions about algebraic groups, Lie groups and Lie algebras, symmetric spaces, lattices and arithmetic groups, virtual cohomological dimension and Bredon cohomology.

\subsection{Algebraic groups and Lie groups}

An \textit{algebraic group} is a subgroup $\mathbb{G}$ of $\SL(N,\mathbb{C})$ determined by a collection of polynomials. It is \textit{defined over} a subfield $k$ of $\mathbb{C}$ if those polynomials can be chosen to have coefficients in $k$. The Galois criterion (see \cite[Prop. 14.2 p.30]{Borel}) says that $\mathbb{G}$ is defined over $k$ if and only if $\mathbb{G}$ is stable under the Galois group $\Gal(\mathbb{C}/k)$. If $\mathbb{G}$ is an algebraic group and $R \subset \mathbb{C}$ is a ring we note $\mathbb{G}_R$ the set of elements of $\mathbb{G}$ with entries in $R$. If $\mathbb{G}$ is an algebraic group defined over $\mathbb{R}$, it is well-known that the groups $\mathbb{G}_{\mathbb{C}}$ and $\mathbb{G}_{\mathbb{R}}$ are Lie groups with finitely many connected components. In fact, $\mathbb{G}$ is Zariski connected if and only if $\mathbb{G}_{\mathbb{C}}$ is a connected Lie group, whereas $\mathbb{G}_{\mathbb{R}}$ may not be connected in this case. A non-abelian algebraic group (or Lie group) is said \textit{simple} if every connected normal subgroup is trivial, and \textit{semisimple} if every connected normal abelian subgroup is trivial. Note that if $\mathbb{G}$ is a semisimple algebraic group defined over $k=\mathbb{R}$ or $\mathbb{C}$ then $\mathbb{G}_k$ is a semisimple Lie group. Any connected semisimple complex linear Lie group is algebraic and any connected semisimple real linear Lie group is the identity component of the group of real points of an algebraic group defined over $\mathbb{R}$. Recall that two Lie groups $G_1$ and $G_2$ are \textit{isogenous} if they are locally isomorphic, meaning that there exist finite normal subgroups $N_1 \subset G_1^0$ and $N_2 \subset G_2^0$ of the identity components of $G_1$ and $G_2$ such that $G_1^0/N_1$ is isomorphic to $G_2^0/N_2$. A semisimple linear Lie group is isogenous to a product of simple Lie groups. 

The center $Z(\mathbb{G})$ of a semisimple algebraic group $\mathbb{G}$ is finite (it is also the case for semisimple linear Lie groups but not for semisimple Lie groups in general) and the quotient $\mathbb{G}/Z(\mathbb{G})$ is again a semisimple algebraic group (see \cite[Thm. 6.8 p.98]{Borel}). Moreover, if $\mathbb{G}$ is defined over $k$ then so is $\mathbb{G}/Z(\mathbb{G})$. 

A connected algebraic group $\mathbb{T}\subset \SL(N,\mathbb{C})$ is a \textit{torus} if it is diagonalizable, meaning there exists $A \in \SL(N, \mathbb{C})$ such that for every $B \in \mathbb{T}$, $ABA^{-1}$ is diagonal. A torus is in particular abelian and isomorphic, as an algebraic group, to a product $\mathbb{C}^*\times \dots \times \mathbb{C}^*$. If $\mathbb{T}$ is defined over $k$, it is said to be \textit{$k$-split} if the conjugating element $A$ can be chosen in $\SL(N,k)$. A torus in an algebraic group $\mathbb{G}$ is a subgroup that is a torus. It is said to be \textit{maximal} if it is not strictly contained in any other torus. An important fact is that any two maximal tori in $\mathbb{G}$ are conjugate in $\mathbb{G}$, and that if $\mathbb{G}$ is defined over $k$, then any two maximal $k$-split tori are conjugate by an element in $\mathbb{G}_k$. The \textit{$k$-rank} of $\mathbb{G}$ (or of $\mathbb{G}_k$), denoted by $\rk_k\mathbb{G}$ (or $\rk_k\mathbb{G}_k$), is the dimension of any maximal $k$-split torus in $\mathbb{G}$, and the \textit{rank} of $\mathbb{G}$ is just the $\mathbb{C}$-rank. 

We refer to \cite{Borel}, \cite{Knapp} and \cite{Onishchik-Vinberg} for basic facts about algebraic groups and Lie groups.

\subsection{Lie algebras and their automorphisms}

Recall that the Lie algebra $\mathfrak{g}$ of a Lie group $G$ is the set of left-invariant vector fields.
A \textit{subalgebra} of $\mathfrak{g}$ is a subspace closed under Lie bracket. An \textit{ideal} is a subalgebra $I$ such that $[\mathfrak{g},I] \subset I$. The Lie algebra $\mathfrak{g}$ is \textit{simple} if it is not abelian and has no non-trivial ideals, and \textit{semisimple} if it has no non-zero abelian ideals. A Lie group is simple (resp. semisimple) if and only if its Lie algebra is simple (resp. semisimple). A semisimple Lie algebra is isomorphic to a finite direct sum of simple ones.

By Lie's third theorem, if $\mathfrak{g}$ is a finite dimensional real Lie algebra (which will be always the case here), there exists a connected Lie group, unique up to covering, whose Lie algebra is $\mathfrak{g}$. This means that their exists a unique simply connected Lie group $G$  associated to $\mathfrak{g}$, and every other connected Lie group whose Lie algebra is $\mathfrak{g}$ is a quotient of $G$ by a subgroup contained in the center. In particular, $G_{ad}=G/Z(G)$ is the unique connected centerfree Lie group associated to $\mathfrak{g}$. The group $G_{ad}$ is called the \textit{adjoint group} of $\mathfrak{g}$. The adjoint group is a linear algebraic group, whereas its universal cover may be not linear (see for instance the universal cover of $\PSL(2,\mathbb{R})$). It follows that the classification of simple Lie algebras is in correspondance with that of simple Lie groups. A Lie algebra is said \textit{compact} if the adjoint group is.

An automorphism of a Lie algebra $\mathfrak{g}$ is a bijective linear endomorphism which preserves the Lie bracket. The group of automorphisms of $\mathfrak{g}$ is denoted $\Aut(\mathfrak{g})$, it is linear and algebraic but not connected in general. If $G$ is a Lie group associated to $\mathfrak{g}$, then the differential of a Lie group automorphism of $G$ is an automorphism of $\mathfrak{g}$. Conversely, if $G$ is either simply connected or connected and centerfree, any automorphism of $\mathfrak{g}$ comes from an automorphism of $G$. In this case, we will often identify these two automorphisms and denote them by the same letter.  An \textit{inner automorphism} is the derivative of the conjugation in $G$ by an element $A  \in G$ - we denote it $\Ad(A)$. The group $\Inn(\mathfrak{g})$ of inner automorphisms is a normal subgroup of $\Aut(\mathfrak{g})$. It is also the identity component of $\Aut(\mathfrak{g})$ and is isomorphic to the adjoint group $G_{ad}$. If $\mathfrak{g}$ is semisimple, the subgroup $\Inn(\mathfrak{g})$ is of finite index in $\Aut(\mathfrak{g})$ and the quotient $\Aut(\mathfrak{g})/\Inn(\mathfrak{g})$ is the (finite) group of \textit{outer automorphisms} $\Out(\mathfrak{g})$. Moreover if $\mathfrak{g}$ is simple, $\Out(\mathfrak{g})$ can be seen as a subgroup of $\Aut(\mathfrak{g})$ and $\Aut(\mathfrak{g})$ is the semidirect product of $\Out(\mathfrak{g})$ and $\Inn(\mathfrak{g})$, that is
$\Aut(\mathfrak{g})=\Inn(\mathfrak{g}) \rtimes \Out(\mathfrak{g})$ (see \cite{Gundogane}).

Note that even if $\mathfrak{g}$ is complex, we let $\Aut(\mathfrak{g})$ be the group of real automorphisms. If $\mathfrak{g}$ is complex and simple then $\Aut(\mathfrak{g})$ contains the complex automorphism group $\Aut_{\mathbb{C}}(\mathfrak{g})$ as a subgroup of index 2, the quotient being generated by complex conjugation (see \cite[Prop. 4.1]{Djokovic}).

Recall that if $\mathfrak{g}$ is a complex Lie algebra, a \textit{real form} of $\mathfrak{g}$ is a real Lie algebra whose complexification is $\mathfrak{g}$. Any real form is the group of fixed points by a \textit{conjugation} of $\mathfrak{g}$, meaning an involutive real automorphism which is antilinear over $\mathbb{C}$.

We refer to \cite{Onishchik-Vinberg}, \cite{Knapp},  \cite{Gundogan} and \cite{Djokovic} for other facts about Lie algebras and their automorphisms.

\subsection{Simple Lie groups, simple Lie algebras and their outer automorphisms}

As mentioned in the previous section, the classification of simple Lie groups (up to isogeny) and of simple Lie algebras are in correspondance. Both are due to by Cartan. We will now see that of simple Lie groups. Every linear simple Lie group is isogenous to either a classical group or to one of the finitely many exceptional groups. We denote the transpose of a matrix $A$ by $A^t$ and its conjugate transpose by $A^*$ and we consider the particular matrices

\[ J_n=\begin{pmatrix}
 & \Id_n \\
-\Id_n  \\
\end{pmatrix}, Q_{p,q}=\begin{pmatrix}
-\Id_p &  \\
 & \Id_q \\
\end{pmatrix}. \]

\noindent The non-compact classical simple Lie groups are the groups in the following list

\vspace{0.2cm}

 $\SL(n,\mathbb{C})=\{A \in \GL(n,\mathbb{C})|\det A=1\} \hfill n \geqslant 2$
 
 $\SO(n,\mathbb{C})=\{A \in \SL(n,\mathbb{C})|A^tA=\Id\} \hfill n\geqslant 3, n\neq 4$
 
 $\Sp(2n,\mathbb{C})=\{A \in \SL(2n,\mathbb{C})|A^tJ_nA=J_n\} \hfill n\geqslant 1$
 
 $\SL(n,\mathbb{R})=\{A \in \GL(n,\mathbb{R})|\det A=1 \} \hfill n\geqslant 2$
 
 $\SL(n,\mathbb{H})=\{A \in \GL(n,\mathbb{H})|\det A=1 \} \hfill n\geqslant 2$

 $\SO(p,q)=\{A \in \SL(p+q,\mathbb{R})|A^tQ_{p,q}A=Q_{p,q}\} \hfill 1\leqslant p \leqslant q, p+q \geqslant 3$
 
 $\SU(p,q)=\{A \in \SL(p+q,\mathbb{C})|A^*Q_{p,q}A=Q_{p,q} \} \hfill 1\leqslant p \leqslant q, p+q \geqslant 3$
 
  $\Sp(p,q)=\{A \in \GL(p+q,\mathbb{H})|A^*Q_{p,q}A=Q_{p,q} \} \hfill 1\leqslant p \leqslant q, p+q \geqslant 3$
 
  $\Sp(2n,\mathbb{R})=\{A \in \SL(2n,\mathbb{R})|A^tJ_nA=J_n \} \hfill n\geqslant 1$
  
  $\SO^*(2n)=\{A \in \SU(n,n)|A^tQ_{n,n}J_nA=Q_{n,n}J_n\} \hfill n\geqslant 2$
  
  \vspace{0.2cm}
  
\noindent Similarly we give the list of the compact ones
 
 \vspace{0.2cm}
  
  $\SO_n=\{A \in \SL(n,\mathbb{R})|A^tA=\Id\} \hfill \mathrm{O}_n=\{A \in \GL(n,\mathbb{R})|A^tA=\Id\}$
  
  $\SU_n=\{A \in \SL(n,\mathbb{C})|A^*A=\Id\} \hfill \U_n=\{A \in \GL(n,\mathbb{C})|A^*A=\Id\}$
  
 $\Sp_n=\{A \in \GL(n,\mathbb{H})|A^*A=\Id\}$
 
 \vspace{0.2cm}
 
\noindent The compact exceptional Lie groups are
\[ G_2, F_4, E_6, E_7, E_8\]
and the non-compact ones are the complex ones (which are the complexifications of the previous compact groups)
\[ G_2^{\mathbb{C}}, F_4^{\mathbb{C}}, E_6^{\mathbb{C}}, E_7^{\mathbb{C}}, E_8^{\mathbb{C}} \]
and their real forms
\[ G_{2(2)}, F_{4(4)}, F_{4(-20)}, F_{4(4)}, E_{6(6)}, E_{6(2)}, E_{6(-14)}, E_{6(-26)}, \]
\[ E_{7(7)}, E_{7(-5)}, E_{7(-25)}, E_{8(8)}, E_{8(-24)}. \]

We refer to \cite{Yokota} for definitions and complete descriptions of the simply connected versions of the exceptional Lie groups. Note that in this paper we will always consider the centerless versions with the same notations.

As usual, the simple Lie algebra associated to a simple Lie group will be denoted by gothic caracters, for instance $\mathfrak{sl}(n,\mathbb{R})$ is the Lie algebra of $\SL(n,\mathbb{R})$. Note that the adjoint group of $\mathfrak{sl}(n,\mathbb{R})$ is $\PSL(n,\mathbb{R})$.
The classification of simple Lie algebras runs in parallel to that of simple Lie groups. The following table summarizes the structure of the outer automorphisms groups of simple Lie algebras (see \cite{Gundogan} section 3.2). We denote by $S_n$ the symmetric group and $D_{2n}$ the dihedral group.

\newpage
\begin{center}
\begin{tabular}{|*{2}{c|}}
\hline
$\mathfrak{g}$ & $\Out(\mathfrak{g})$ \\
\hline
$\mathfrak{sl}(n,\mathbb{C}), ~ n\geqslant 3$ & $\mathbb{Z}_2 \times \mathbb{Z}_2$ \\
\hline
$\mathfrak{so}(8,\mathbb{C})$ & $\mathcal{S}_3 \times \mathbb{Z}_2$ \\
\hline
$\mathfrak{so}(2n,\mathbb{C}),~ n \geqslant 5$ & $\mathbb{Z}_2 \times \mathbb{Z}_2$ \\
\hline
$\mathfrak{e}_6^{\mathbb{C}}$ & $\mathbb{Z}_2 \times \mathbb{Z}_2$ \\
\hline
all others complex Lie algebras & $\mathbb{Z}_2$ \\
\hline
$\mathfrak{sl}(2,\mathbb{R})$ & $\mathbb{Z}_2$ \\
\hline
$\mathfrak{sl}(n,\mathbb{R}),~n \geqslant 3$ odd & $\mathbb{Z}_2$ \\
\hline
$\mathfrak{sl}(n,\mathbb{R}),~n \geqslant 4$ even & $\mathbb{Z}_2 \times \mathbb{Z}_2$ \\
\hline
$\mathfrak{su}(p,q),~p \neq q$ & $\mathbb{Z}_2$ \\
\hline
$\mathfrak{su}(p,p),~p \geqslant 2$ & $\mathbb{Z}_2 \times \mathbb{Z}_2$ \\
\hline
$\mathfrak{sl}(n,\mathbb{H})$ & $\mathbb{Z}_2$\\
\hline
$\mathfrak{so}(p,q),~ p+q$ odd & $\mathbb{Z}_2$ \\
\hline
$\mathfrak{so}(p,q)$, $p$ and $q$ odd, $p\neq q$ & $\mathbb{Z}_2$ \\
\hline
$\mathfrak{so}(p,q)$ $p$ and $q$ even, $p\neq q$ & $\mathbb{Z}_2 \times \mathbb{Z}_2$ \\
\hline
$\mathfrak{so}(p,p)$ $p \geqslant 5$ odd & $\mathbb{Z}_2 \times \mathbb{Z}_2$ \\
\hline
$\mathfrak{so}(p,p)$ $p \geqslant 6$ even & $\mathcal{D}_4$ \\
\hline
$\mathfrak{so}(4,4)$ & $\mathcal{S}_4$ \\
\hline
$\mathfrak{sp}(2n,\mathbb{R})$ & $\mathbb{Z}_2$ \\
\hline
$\mathfrak{sp}(p,p)$ & $\mathbb{Z}_2$ \\
\hline
$\mathfrak{so}^*(2n)$ & $\mathbb{Z}_2$ \\
\hline
$\mathfrak{e}_{6(j)}$, $j=6,2,-14,-26$ & $\mathbb{Z}_2$ \\
\hline
$\mathfrak{e}_{7(j)}$, $j=7,-25$ & $\mathbb{Z}_2$ \\
\hline
all others real Lie algebras & 1 \\
\hline
\end{tabular}
\captionof{table}{Outer automorphisms groups of simple Lie algebras}
\end{center}

Note that we have the isomorphisms: $\mathfrak{so}(3,\mathbb{C}) \cong \mathfrak{sl}(2,\mathbb{C})\cong \mathfrak{sp}(2,\mathbb{C})$, $\mathfrak{so}(5, \mathbb{C}) \cong \mathfrak{sp}(4,\mathbb{C})$, $\mathfrak{so}(6,\mathbb{C}) \cong \mathfrak{sl}(4,\mathbb{C})$ and the corresponding ones between their real forms, and that $\mathfrak{so}(4,\mathbb{C})$ is not simple but isomorphic to $\mathfrak{sl}(2,\mathbb{C}) \oplus \mathfrak{sl}(2,\mathbb{C})$. 

\subsection{Symmetric spaces}

 Let $G$ be a Lie group. A \textit{symmetric space} is a space of the form $G/G^{\rho}$ where $\rho$ is an involutive automorphism of $G$ and $G^{\rho}$ its fixed points set. It is said \textit{irreducible} if it can not be decomposed as a product. From an algebraic point of view, the irreducibility of $G/H$ implies that the Lie algebra $\mathfrak{h}$ of $H$ is a maximal subalgebra of the Lie algebra $\mathfrak{g}$ of $G$. Equivalently, the irreducibility of $G/H$ implies that the identity component of $H$ is a maximal connected Lie subgroup of the identity component of $G$. 

Another point of view on symmetric spaces is based on Lie algebras. If $G/G^{\rho}$ is a symmetric space and $\mathfrak{g}$ is the Lie algebra of $G$, the involutive automorphism $\rho$ induces an involutive automorphism of $\mathfrak{g}$ whose fixed point set is the Lie algebra $\mathfrak{h}$ of $H=G^{\rho}$. We can thus always associate to a symmetric space $G/H$ a linear space $\mathfrak{g}/\mathfrak{h}$, called a \textit{local symmetric space}. The Lie subalgebra $\mathfrak{h}$ is called the \textit{isotropy algebra} of $\mathfrak{g}/\mathfrak{h}$ (more generally we say that $\mathfrak{h}$ is an isotropy algebra if it is the fixed point set of an involutive automorphism) . Conversely, if $\mathfrak{g}$ is a Lie algebra, $G$ a simply connected or connected and centerless Lie group whose Lie algebra is $\mathfrak{g}$, and $\mathfrak{h} \subset \mathfrak{g}$ an isotropy algebra, the local symmetric space $\mathfrak{g}/\mathfrak{h}$ lifts to a symmetric space $G/H$, because $\Aut(\mathfrak{g})=\Aut(G)$. So the classification of symmetric spaces are in correspondance with those of local symmetric spaces, and has been done by Berger in \cite{Berger}.

Note that if $\mathfrak{g}$ is simple and complex, and if $\rho \in \Aut(\mathfrak{g})$ is an involution, then $\rho$ is either $\mathbb{C}$-linear and in this case $\mathfrak{h}=\mathfrak{g}^{\rho}$ is also complex, or $\rho$ is anti-linear (that means it is a conjugation) and $\mathfrak{h}=\mathfrak{g}^{\rho}$ is a real form of $\mathfrak{g}$. Note also that is $\mathfrak{g}$ is real and $\rho$ is an involution, then $\rho$ can be extended to a $\mathbb{C}$-linear involution $\rho^{\mathbb{C}}$ of the complexification $\mathfrak{g}^{\mathbb{C}}$ of $\mathfrak{g}$, and the isotropy algebra $(\mathfrak{g}^{\mathbb{C}})^{\rho^{\mathbb{C}}}$ is the complexification of $\mathfrak{h}=\mathfrak{g}^{\rho}$, that is $(\mathfrak{g}^{\mathbb{C}})^{\rho^{\mathbb{C}}}=\mathfrak{h}^{\mathbb{C}}$.

We give now the list of the non-compact isotropy algebras of the local symmetric spaces associated to $\mathfrak{sl}(n,\mathbb{C})$ and its real forms. 

\begin{center}
\begin{tabular}{|*{4}{c|}}
\hline
$\mathfrak{g}^{\mathbb{C}}=\mathfrak{sl}(n,\mathbb{C})$ & $\mathfrak{h}^{\mathbb{C}}=\mathfrak{so}(n,\mathbb{C})$ & $\mathfrak{h}^{\mathbb{C}}=\mathfrak{s(gl(}k,\mathbb{C}) \oplus \mathfrak{gl}(l,\mathbb{C}))$ & $\mathfrak{h}^{\mathbb{C}}=\mathfrak{sp}(n,\mathbb{C})$ \\
\hline
$\mathfrak{g}=\mathfrak{sl}(n,\mathbb{R})$ & $\mathfrak{h}=\mathfrak{so}(k,l)$
 \multirow{2}{*} & $\mathfrak{h}=\mathfrak{s(gl(}k,\mathbb{R}) \oplus \mathfrak{gl}(l,\mathbb{R}))$  & $\mathfrak{h}=\mathfrak{sp}(n,\mathbb{R})$ \\
& & $\mathfrak{h}=\mathfrak{gl}\left(\frac{n}{2},\mathbb{C}\right)$&
 \\
\hline
 $\mathfrak{g}=\mathfrak{su}(p,q)$ 
\multirow{2}{*}  & $\mathfrak{h}=\mathfrak{so}(p,q)$ & $\mathfrak{h}=\mathfrak{s(u(}k_p,k_q) \oplus \mathfrak{u}(l_p,l_q))$ & $\mathfrak{h}=\mathfrak{sp}\left(\frac{p}{2},\frac{q}{2}\right)$ \\
 & ($p=q$) $\mathfrak{h}=\mathfrak{so}^*(2p)$ & $\mathfrak{h}=\mathfrak{gl}(p,\mathbb{C})$ & $\mathfrak{h}=\mathfrak{sp}(2p,\mathbb{R})$
\\
\hline
$\mathfrak{g}=\mathfrak{sl}\left(\frac{n}{2},\mathbb{H}\right)$
\multirow{2}{*} & $\mathfrak{h}=\mathfrak{so}^*\left(2\left(\frac{n}{2}\right)\right)$ & $\mathfrak{h}=\mathfrak{s(gl(}\frac{k}{2},\mathbb{H}) \oplus \mathfrak{gl}(\frac{l}{2},\mathbb{H}))$ & $\mathfrak{h}=\mathfrak{sp}\left(\frac{k}{2},\frac{l}{2}\right)$ \\
& & $\mathfrak{gl}\left(\frac{n}{2},\mathbb{C}\right)$ & 
\\
\hline
\end{tabular}
\captionof{table}{Non-compact isotropy algebras of $\mathfrak{sl}(n,\mathbb{C})$ and its real forms}
\end{center}

Table 2 is organized as follows: in the first line we give the complex isotropy algebras $\mathfrak{h}^{\mathbb{C}}$ of $\mathfrak{sl}(n,\mathbb{C})$ (fixed by a complex involution). Each column consists of real forms of the complex algebra in the first entry. The local symmetric spaces associated to $\mathfrak{sl}(n,\mathbb{C})$ are then those of the form $\mathfrak{sl}(n,\mathbb{C})/\mathfrak{h}^{\mathbb{C}}$, for instance $\mathfrak{sl}(n,\mathbb{C})/\mathfrak{so}(n,\mathbb{C})$, or of the form $\mathfrak{sl}(n,\mathbb{C})/\mathfrak{g}$, for instance $\mathfrak{sl}(n,\mathbb{C})/\mathfrak{sl}(n,\mathbb{R})$. The ones associated to a real form $\mathfrak{g}$ are of the form $\mathfrak{g}/\mathfrak{h}$, for instance $\mathfrak{sl}(n,\mathbb{R})/\mathfrak{so}(k,l)$ with $k+l=n$. 

The following tables summarize the classification for other simple Lie algebras. They are organized in a similar way.

\begin{center}
\begin{tabular}{|*{3}{c|}}
\hline
$\mathfrak{g}^{\mathbb{C}}=\mathfrak{so}(n,\mathbb{C})$ & $\mathfrak{h}^{\mathbb{C}}=\mathfrak{so}(k,\mathbb{C}) \oplus \mathfrak{so}(l,\mathbb{C})$ & $\mathfrak{h}^{\mathbb{C}}=\mathfrak{gl}\left(\frac{n}{2},\mathbb{C}\right)$ \\
\hline
$\mathfrak{g}=\mathfrak{so}(p,q)$ 
\multirow{2}{*} & $\mathfrak{h}=\mathfrak{so}(k_p,k_q) \oplus \mathfrak{so}(l_p,l_q))$ & $\mathfrak{h}=\mathfrak{u}\left(\frac{p}{2},\frac{q}{2}\right)$ \\
& ($p=q$) $\mathfrak{h}=\mathfrak{so}(p,\mathbb{C})$ & $\mathfrak{h}=\mathfrak{gl}(p,\mathbb{R})$ \\
\hline
$\mathfrak{g}=\mathfrak{so}^*\left(2\left(\frac{n}{2}\right)\right)$
\multirow{2}{*} & $\mathfrak{h}=\mathfrak{so}^*\left(2\left(\frac{k}{2}\right)\right) \oplus \mathfrak{so}^*\left(2\left(\frac{l}{2}\right)\right)$ & $\mathfrak{h}=\mathfrak{u}\left(\frac{k}{2},\frac{l}{2}\right)$ \\
& $\mathfrak{h}=\mathfrak{so}\left(\frac{n}{2},\mathbb{C}\right)$ & $\mathfrak{h}=\mathfrak{gl}\left(\frac{n}{4},\mathbb{H}\right)$ \\
\hline
\end{tabular}
\captionof{table}{Non-compact isotropy algebras of $\mathfrak{so}(n,\mathbb{C})$ and its real forms}
\end{center}

\begin{center}
\begin{tabular}{|*{3}{c|}}
\hline
$\mathfrak{g}^{\mathbb{C}}=\mathfrak{sp}(2n,\mathbb{C})$ & $\mathfrak{h}^{\mathbb{C}}=\mathfrak{sp}(2k,\mathbb{C}) \oplus \mathfrak{sp}(2l,\mathbb{C})$ & $\mathfrak{h}^{\mathbb{C}}=\mathfrak{gl}(n,\mathbb{C})$ \\
\hline
$\mathfrak{g}=\mathfrak{sp}(p,q)$
\multirow{2}{*} & $\mathfrak{h}=\mathfrak{sp}(k_p,k_q) \oplus \mathfrak{sp}(l_p,l_q)$ & $\mathfrak{h}=\mathfrak{u}(p,q)$ \\
& $(p=q)$ $\mathfrak{h}=\mathfrak{sp}(p,\mathbb{C})$ & $\mathfrak{h}=\mathfrak{gl}(p,\mathbb{H})$ \\
\hline
$\mathfrak{g}=\mathfrak{sp}(2n,\mathbb{R})$
\multirow{2}{*} & $\mathfrak{h}=\mathfrak{sp}(2k,\mathbb{R}) \oplus \mathfrak{sp}(2l,\mathbb{R})$ & $\mathfrak{h}=\mathfrak{u}(k,l)$ \\
& $\mathfrak{h}=\mathfrak{sp}(n,\mathbb{C})$ & $\mathfrak{h}=\mathfrak{gl}(n,\mathbb{R})$ \\
\hline
\end{tabular}
\captionof{table}{Non-compact isotropy algebras of $\mathfrak{sp}(2n,\mathbb{C})$ and its real forms}
\end{center}

\begin{center}
\begin{tabular}{|*{2}{c|}}
\hline
$\mathfrak{g}^{\mathbb{C}}=\mathfrak{g}_2^{\mathbb{C}}$ & $\mathfrak{h}^{\mathbb{C}}=\mathfrak{sl}(2,\mathbb{C}) \oplus \mathfrak{sl}(2,\mathbb{C})$ \\
\hline
$\mathfrak{g}=\mathfrak{g}_{2(2)}$ & $\mathfrak{h}=\mathfrak{sl}(2,\mathbb{R}) \oplus \mathfrak{sl}(2,\mathbb{R})$ \\
\hline
\end{tabular}
\captionof{table}{Non-compact isotropy algebras of $\mathfrak{g}_2^{\mathbb{C}}$ and its real form}
\end{center}

\begin{center}
\begin{tabular}{|*{3}{c|}}
\hline
$\mathfrak{g}^{\mathbb{C}}=\mathfrak{f}_4^{\mathbb{C}}$ & $\mathfrak{h}^{\mathbb{C}}=\mathfrak{sp}(6,\mathbb{C}) \oplus \mathfrak{sp}(2,\mathbb{C})$ & $\mathfrak{h}^{\mathbb{C}}=\mathfrak{so}(9,\mathbb{C})$ \\
\hline
$\mathfrak{g}=\mathfrak{f}_{4(4)}$
\multirow{2}{*} & $\mathfrak{h}=\mathfrak{sp}(6,\mathbb{R}) \oplus \mathfrak{sp}(2,\mathbb{R})$ & $\mathfrak{h}=\mathfrak{so}(4,5)$ \\
& $\mathfrak{h}=\mathfrak{sp}(1,2) \oplus \mathfrak{sp}(1)$ & \\
\hline
$\mathfrak{g}=\mathfrak{f}_{4(-20)}$ & $\mathfrak{h}=\mathfrak{sp}(1,2) \oplus \mathfrak{sp}(1)$ & $\mathfrak{h}=\mathfrak{so}(1,8)$ \\
\hline
\end{tabular}
\captionof{table}{Non-compact isotropy algebras of $\mathfrak{f}_4^{\mathbb{C}}$ and its real forms}
\end{center}

\begin{center}
\begin{tabular}{|*{5}{c|}}
\hline
$\mathfrak{g}^{\mathbb{C}}=\mathfrak{e}_6^{\mathbb{C}}$ & $\mathfrak{h}^{\mathbb{C}}=\mathfrak{sp}(8,\mathbb{C})$ &
$\mathfrak{h}^{\mathbb{C}}=\mathfrak{sl}(6,\mathbb{C}) \oplus \mathfrak{sl}(2,\mathbb{C})$ &
$\mathfrak{h}^{\mathbb{C}}=\mathfrak{so}(10,\mathbb{C}) \oplus \mathfrak{so}(2,\mathbb{C})$ & $\mathfrak{h}^{\mathbb{C}}=\mathfrak{f}_4^{\mathbb{C}}$ \\
\hline
$\mathfrak{g}=\mathfrak{e}_{6(6)}$
\multirow{2}{*} & $\mathfrak{h}=\mathfrak{sp}(2,2)$ &
$\mathfrak{h}=\mathfrak{sl}(6,\mathbb{R}) \oplus \mathfrak{sl}(2,\mathbb{R})$ &
$\mathfrak{h}=\mathfrak{so}(5,5) \oplus \mathfrak{so}(1,1)$ & $\mathfrak{h}=\mathfrak{f}_{4(4)}$ \\
& $\mathfrak{h}=\mathfrak{sp}(8,\mathbb{R})$ & $\mathfrak{h}=\mathfrak{sl}(3,\mathbb{H}) \oplus \mathfrak{su}(2)$ & & \\
\hline
$\mathfrak{g}=\mathfrak{e}_{6(2)}$
\multirow{2}{*} & $\mathfrak{h}=\mathfrak{sp}(1,3)$ &
$\mathfrak{h}=\mathfrak{su}(2,4) \oplus \mathfrak{su}(2)$ &
$\mathfrak{h}=\mathfrak{so}(4,6) \oplus \mathfrak{so}(2)$ & $\mathfrak{h}=\mathfrak{f}_{4(4)}$ \\
& $\mathfrak{h}=\mathfrak{sp}(8,\mathbb{R})$ & $\mathfrak{h}=\mathfrak{su}(3,3) \oplus \mathfrak{sl}(2,\mathbb{R})$ & $\mathfrak{h}=\mathfrak{so}^*(10) \oplus \mathfrak{so}(2) $ & \\
\hline
$\mathfrak{g}=\mathfrak{e}_{6(-14)}$
\multirow{2}{*} & $\mathfrak{h}=\mathfrak{sp}(2,2)$ &
$\mathfrak{h}=\mathfrak{su}(2,4) \oplus \mathfrak{su}(2)$ &
$\mathfrak{h}=\mathfrak{so}(2,8) \oplus \mathfrak{so}(2)$ & $\mathfrak{h}=\mathfrak{f}_{4(-20)}$ \\
& & $\mathfrak{h}=\mathfrak{su}(1,5) \oplus \mathfrak{sl}(2,\mathbb{R})$ & $\mathfrak{h}=\mathfrak{so}^*(10) \oplus \mathfrak{so}(2) $ & \\
\hline
$\mathfrak{g}=\mathfrak{e}_{6(-26)}$ & $\mathfrak{h}=\mathfrak{sp}(1,3)$ &
$\mathfrak{h}=\mathfrak{sl}(3,\mathbb{H}) \oplus \mathfrak{sp}(1)$ &
$\mathfrak{h}=\mathfrak{so}(1,9) \oplus \mathfrak{so}(1,1)$ & $\mathfrak{h}=\mathfrak{f}_{4(-20)}$ \\
\hline
\end{tabular}
\captionof{table}{Non-compact isotropy algebras of $\mathfrak{e}_6^{\mathbb{C}}$ and its real forms}
\end{center}

\begin{center}
\begin{tabular}{|*{4}{c|}}
\hline
$\mathfrak{g}^{\mathbb{C}}=\mathfrak{e}_7^{\mathbb{C}}$ & $\mathfrak{h}^{\mathbb{C}}=\mathfrak{sl}(8,\mathbb{C})$ &
$\mathfrak{h}^{\mathbb{C}}=\mathfrak{so}(12,\mathbb{C}) \oplus \mathfrak{sl}(2,\mathbb{C})$ & $\mathfrak{h}^{\mathbb{C}}=\mathfrak{e}_6^{\mathbb{C}} \oplus \mathfrak{so}(2,\mathbb{C})$ \\
\hline
$\mathfrak{g}=\mathfrak{e}_{7(7)}$
\multirow{3}{*} & $\mathfrak{h}=\mathfrak{su}(4,4)$ &
$\mathfrak{h}=\mathfrak{so}(6,6) \oplus \mathfrak{sl}(2,\mathbb{R})$ & $\mathfrak{h}=\mathfrak{e}_{6(6)} \oplus \mathfrak{so}(1,1)$ \\
& $\mathfrak{h}=\mathfrak{sl}(8,\mathbb{R})$  & $\mathfrak{h}=\mathfrak{so}^*(12) \oplus \mathfrak{sp}(1)$ & $\mathfrak{h}=\mathfrak{e}_{6(2)} \oplus \mathfrak{so}(2)$ \\
& $\mathfrak{h}=\mathfrak{sl}(4,\mathbb{H})$ & & \\
\hline
$\mathfrak{g}=\mathfrak{e}_{7(-5)}$
\multirow{2}{*} & $\mathfrak{h}=\mathfrak{su}(4,4)$ &
$\mathfrak{h}=\mathfrak{so}(4,8) \oplus \mathfrak{su}(2)$ & $\mathfrak{h}=\mathfrak{e}_{6(2)} \times \mathfrak{so}(2)$ \\
& $\mathfrak{h}=\mathfrak{su}(2,6)$  & $\mathfrak{h}=\mathfrak{so}^*(12) \oplus \mathfrak{sl}(2,\mathbb{R}) $ & $\mathfrak{h}=\mathfrak{e}_{6(-14)} \oplus \mathfrak{so}(2)$  \\
\hline
$\mathfrak{g}=\mathfrak{e}_{7(-25)}$
\multirow{2}{*} & $\mathfrak{h}=\mathfrak{sl}(4,\mathbb{H})$  &
$\mathfrak{h}=\mathfrak{so}(2,10) \oplus \mathfrak{sl}(2,\mathbb{R})$ & $\mathfrak{h}=\mathfrak{e}_{6(-14)} \oplus \mathfrak{so}(2)$ \\
& $\mathfrak{h}=\mathfrak{su}(2,6)$ & $\mathfrak{h}=\mathfrak{so}^*(12) \oplus \mathfrak{sp}(1)$ & $\mathfrak{h}=\mathfrak{e}_{6(-26)} \oplus \mathfrak{so}(1,1)$ \\
\hline
\end{tabular}
\captionof{table}{Non-compact isotropy algebras of $\mathfrak{e}_7^{\mathbb{C}}$ and its real forms}
\end{center}

\begin{center}
\begin{tabular}{|*{3}{c|}}
\hline
$\mathfrak{g}^{\mathbb{C}}=\mathfrak{e}_8^{\mathbb{C}}$ & $\mathfrak{h}^{\mathbb{C}}=\mathfrak{e}_7^{\mathbb{C}} \oplus \mathfrak{sl}(2,\mathbb{C})$ & $\mathfrak{h}^{\mathbb{C}}=\mathfrak{so}(16,\mathbb{C})$ \\
\hline
$\mathfrak{g}=\mathfrak{e}_{8(8)}$
\multirow{2}{*} & $\mathfrak{h}=\mathfrak{e}_{7(7)} \oplus \mathfrak{sl}(2,\mathbb{R})$ & $\mathfrak{h}=\mathfrak{so}(8,8)$ \\
& $\mathfrak{h}=\mathfrak{e}_{7(-5)} \oplus \mathfrak{su}(2)$ & $\mathfrak{h}=\mathfrak{so}^*(16)$ \\
\hline
$\mathfrak{g}=\mathfrak{e}_{8(-24)}$  
\multirow{2}{*} & $\mathfrak{h}=\mathfrak{e}_{7(-25)} \oplus \mathfrak{sl}(2,\mathbb{R})$ & $\mathfrak{h}=\mathfrak{so}(4,12)$ \\
& $\mathfrak{h}=\mathfrak{e}_{7(-5)} \oplus \mathfrak{su}(2)$ & $\mathfrak{h}=\mathfrak{so}^*(16)$ \\
\hline
\end{tabular}
\captionof{table}{Non-compact isotropy algebras of $\mathfrak{e}_8^{\mathbb{C}}$ and its real forms}
\end{center}

Note that not all the symmetric spaces given in these tables are irreducible. For instance $\mathfrak{sl}(n,\mathbb{C})/\mathfrak{s(gl(}k,\mathbb{C}) \oplus \mathfrak{gl}(l,\mathbb{C}))$ is not. The results of \cite{Berger} are more precise and we refer to them for the list of the irreducible symmetric spaces and the non-irreducibles ones.

We refer to \cite{Helgason} and \cite{Berger} for facts about symmetric spaces and local symmetric spaces.

\subsection{Riemannian symmetric spaces}

We stress that the symmetric spaces $G/H$ associated to the isotropy subalgebras $\mathfrak{h}$ of $\mathfrak{g}$ in Tables 2 to 9 are non-Riemannian. We discuss now a few features about Riemannian symmetric spaces, which are of the form $G/G^{\rho}$ with $G^{\rho}$ compact. The symmetric spaces which are Riemannian spaces of non-positive curvature are called symmetric spaces \textit{of non-compact type}. They are all of the form $S=G/K$ where $G=\Isom(S)^0$ and $K$ is a maximal compact subgroup. If it has no euclidean factors, then $G$ is semisimple, linear and centerless.

 Recall that if $G$ is a Lie group, all maximal compact subgroups are conjugated. If $G$ is semisimple, or more generally reductive, and if $K$ is a maximal compact subgroup, then the symmetric space $G/K$ is called the \textit{Riemannian symmetric space associated to $G$}. It follows that we can identify the smooth manifold
\[ S=G/K \]
with the set of all maximal compact subgroups of $G$. Remark that isogenous Lie groups have isometric associated Riemannian symmetric spaces. In particular, if $G$ is a semisimple linear Lie group, the associated Riemannian symmetric space is the same as that associated to its identity component $G^0$ or to $G/Z(G)$. We can thus assume that $G$ is connected and centerless. In this case, as the image of a maximal compact subgroup by an automorphism of $G$ is again a maximal compact subgroup, we have an action of $\Aut(G)=\Aut(\mathfrak{g})$ by isometries on $S=G/K$. Finally we have that the group of isometries of a symmetric space $S=G/K$ of non-compact type without euclidean factors is $\Aut(\mathfrak{g})$ where $\mathfrak{g}$ is the Lie algebra of $G$.

An important part of our work will be to compute dimensions of fixed point sets
\[ S^{\alpha}=\{ x \in S ~|~ \alpha(x)=x\}  \]
where $\alpha \in \Isom(S)=\Aut(\mathfrak{g})$. Assuming that $G$ is connected and centerless, the fixed point set $S^{\alpha}$ is the Riemannian symmetric space associated to $G^{\alpha}$ (recall that we denote by the same letter the automorphism of $\mathfrak{g}$ and that of $G$). If $A \in G$ we will denote by $S^A$ the fixed point set of the inner automorphism $\Ad(A)$. In the case where $A$ is of finite order, it can be conjugated in the maximal compact subgroup $K$. Then the fixed point set of $G$ by $\Ad(A)$ is the centralizer of $A$ in $G$, that is  $G^{\Ad(A)}=C_G(A)=\{B \in G ~|~ AB=BA \}$. A maximal compact subgroup of $C_G(A)$ is $C_K(A)$, the centralizer of $A$ in $K$. So we can identify $S^A$ with $C_G(A)/C_K(A)$, and we can write:
\[ \dim S^A=\dim C_G(A)-\dim C_K(A). \]

We refer to \cite{Helgason} for other facts about Riemannian symmetric spaces.

\subsection{Lattices and arithmetic groups}

A discrete subgroup $\Gamma$ of a Lie group $G$ is said to be a \textit{lattice} if the quotient $\Gamma \backslash G$ has finite Haar measure. It is said \textit{uniform} (or cocompact) if this quotient is compact and \textit{non-uniform} otherwise. The Borel density theorem (see \cite[Cor. 4.5.6]{Witte}) says that if $G$ is the group of real points of a connected semisimple algebraic group defined over $\mathbb{R}$, and if a lattice $\Gamma \subset G$ projects densely into the maximal compact factor of $G$, then $\Gamma$ is Zariski-dense in $G$. For instance, if $\mathbb{G}$ is a connected semisimple algebraic group defined over $\mathbb{Q}$, then the group $\mathbb{G}_{\mathbb{Z}}$ is a lattice in $\mathbb{G}_{\mathbb{R}}$ and thus Zariski-dense. The group $\mathbb{G}_{\mathbb{Z}}$ is the paradigm of an arithmetic group, which will be defined now.

Let $G$ be a semisimple Lie group with identity component $G^0$ and $\Gamma \subset G$ a lattice. The lattice $\Gamma$ is said to be \textit{arithmetic} if there are a connected algebraic group $\mathbb{G}$ defined over $\mathbb{Q}$, compact normal subgroups $K \subset G^0$, $K' \subset \mathbb{G}_{\mathbb{R}}^0$ and a Lie group isomorphism 
\[ \varphi: G^0/K \rightarrow \mathbb{G}_{\mathbb{R}}^0/K', \]
such that $\varphi(\bar{\Gamma})$ is commensurable to $\bar{\mathbb{G}_{\mathbb{Z}}}$, where $\bar{\Gamma}$ and $\bar{\mathbb{G}_{\mathbb{Z}}}$ are the images of $\Gamma\cap G^0$ and $\mathbb{G}_{\mathbb{Z}}\cap \mathbb{G}_{\mathbb{R}}^0 $ in  $G^0/K$ and $\mathbb{G}_{\mathbb{R}}^0/K$ (recall that two subgroups $H$ and $H'$ of $G$ are commensurable if their intersection is of finite index in both subgroups).

We say that the lattice $\Gamma \subset G$ is \textit{irreducible} if $\Gamma N$ is dense in $G$ for every non-compact, closed, normal subgroup $N$ of $G^0$. The Margulis arithmeticity theorem (see \cite[Ch. IX]{Margulis} and \cite[Thm. 5.2.1]{Witte}) tells us that in a way, most irreducible lattices are arithmetic.

\begin{thm}[Margulis arithmeticity theorem]\label{Théorème Margulis}
Let $G$ be the group of real points of a semisimple algebraic group defined over $\mathbb{R}$ and $\Gamma \subset G$ an irreducible lattice. If $G$ is not isogenous to $\SO(1,n) \times K$ or $\SU(1,n) \times K$ for any compact group $K$, then $\Gamma$ is arithmetic.
\end{thm}

Observe that $\SO(1,n) \times K$ and $\SU(1,n) \times K$ have real rank 1, so the arithmeticity theorem applies to every irreducible lattice in a group of real rank at least 2.

The definition of arithmeticity can be simplified in some cases. If $G$ is connected, centerfree and has no compact factors, the compact subgroup $K$ in the definition must be trivial. Moreover, if $\Gamma$ is non-uniform and irreducible, then the compact subgroup $K'$ is not needed either (see \cite[Cor. 5.3.2]{Witte}). Under the same assumptions, we can also assume that the algebraic group $\mathbb{G}$ is centerfree, and in this case the commensurator of $\mathbb{G}_{\mathbb{Z}}$ in $\mathbb{G}$ is $\mathbb{G}_{\mathbb{Q}}$ and $\varphi(\Gamma) \subset \mathbb{G}_{\mathbb{Q}}$. Under the same hypotheses on $G$, if $\Gamma$ is non irreducible, it is almost a product of irreducible lattices. In fact (see \cite[Prop. 4.3.3]{Witte}), there is a direct decomposition $G=G_1 \times \dots \times G_r$ such that $\Gamma$ is commensurable to $\Gamma_i \times \dots \times \Gamma_r$ where $\Gamma_i=\Gamma \cap G_i$ is an irreducible lattice in $G_i$.

The rational rank of the arithmetic group $\Gamma$, denoted by $\rk_{\mathbb{Q}} \Gamma$, is by definition the $\mathbb{Q}$-rank of the algebraic group $\mathbb{G}$ in the definition of arithmeticity, and we have 
\[ \rk_{\mathbb{Q}} \Gamma \leqslant \rk_{\mathbb{R}} \mathbb{G}. \]
Note that $\rk_{\mathbb{Q}} \Gamma=0$ if and only if $\Gamma$ is cocompact (see \cite[Thm. 8.4]{Borell}).

We refer to \cite{Borell} and \cite{Witte} for other facts about lattices and arithmetic groups. 

\subsection{Virtual cohomological dimension and proper geometric dimension}

Recall that the virtual cohomological dimension of a virtually torsion-free discrete subgroup $\Gamma$ is the cohomological dimension of any torsion-free subgroup $\Gamma'$ of finite index of $\Gamma$, that is
\[ \vcd(\Gamma)=\cd(\Gamma')=\max\{n ~|~ H^n(\Gamma',A) \neq 0 \text{ for a certain $\mathbb{Z}\Gamma'$-module $A$}\}. \]

If $X$ is a cocompact model for $\underline{E}\Gamma$, we can compute the virtual cohomological dimension of $\Gamma$ as
\begin{equation}
\label{Expression vcd}
\vcd(\Gamma)=\max\{n \in \mathbb{N} ~ | ~ \Ho_c^n(X) \neq 0 \}
\end{equation} 
where $\Ho_c^n(X)$ denotes the compactly supported cohomology of $X$ (see \cite[Cor. VIII.7.6]{Brown}).

The proper geometric dimension $\underline{\gd}(\Gamma)$ is the smallest possible dimension of a model for $\underline{E}\Gamma$. 

If $G$ is the group of real points of a semisimple algebraic group, $K\subset G$ a maximal compact subgroup, $S=G/K$ the associated Riemannian symmetric space and $\Gamma \subset G$ a uniform lattice of $G$, $S$ is a model for $\underline{E}\Gamma$ and has dimension $\vcd(\Gamma)$, so we have $\vcd(\Gamma)=\underline{\gd}(\Gamma)$, that is why we will be mostly interested in non-uniform lattices.

We will also rule out the case when the adjoint group $G_{ad}$ of $\mathfrak{g}$ has real rank 1, in fact we have the following (see \cite[Cor. 2.8]{Aramayona})

\begin{prop}\label{Prop rang 1}
Let $\mathbb{G}$ be an algebraic group defined over $\mathbb{R}$ and $\Gamma \subset \mathbb{G}_{\mathbb{R}}$ a lattice. If $\rk_{\mathbb{R}} \mathbb{G}=1$, then $\vcd(\Gamma)=\underline{\gd}(\Gamma)$.
\end{prop}

For the case of higher real rank, recall that by Margulis arithmeticity theorem, $\Gamma$ is arithmetic as long as it is irreducible. 

If $\Gamma$ is non-uniform, $\Gamma \setminus S$ is not compact. However Borel and Serre constructed in \cite{Borel-Serre} a $\Gamma$-invariant bordification of $S$ called the \textit{Borel-Serre bordification} $X$ which is a cocompact model for $\underline{E}\Gamma$ (see \cite[Th. 3.2]{Ji}).

Using their bordification, Borel and Serre proved in \cite{Borel-Serre} the following theorem which links the virtual cohomological dimension and the rational rank of such an arithmetic lattice.

\begin{thm}[Borel-Serre]\label{Theoreme Borel-Serre}
Let $G$ be a semisimple Lie group, $K \subset G$ a maximal compact subgroup and $\Gamma \subset G$ an arithmetic lattice. Then: \[ \vcd(\Gamma) = \dim(G/K)-\rk_{\mathbb{Q}} \Gamma. \]
In particular: \[ \vcd(\Gamma) \geqslant \dim(G/K)-\rk_{\mathbb{R}} G. \]
\end{thm}

Before moving on, note that we will often in this article consider groups up to isogeny, and the philosophy behind it is that normal finite subgroups do not change the dimensions, indeed we have:

\begin{lemma}\label{Finite normal subgroups}
Let $\Gamma$ be an infinite discrete group, and $N \subset \Gamma$ a finite normal subgroup. Then:
\[ \underline{\gd}(\Gamma)=\underline{\gd}\left(\Gamma/N\right), \]
\[ \vcd(\Gamma)=\vcd\left(\Gamma/N\right). \]
\end{lemma}

\begin{demo}
For the first equality: if $X$ is a model for $\underline{E}\Gamma$ it follows easily that $X^N$ is a model for $\underline{E}(\Gamma/N)$ of dimension lower than those of $X$, so:
\[ \underline{\gd}(\Gamma) \geqslant \underline{\gd}\left(\Gamma/N\right). \]
Reciprocally, a model for $\underline{E}(\Gamma/N)$ is also a model for $\underline{E}\Gamma$ and we have the other inequality.

For the second equality, it suffices to recall that $\vcd(\Gamma)=\cd(\Gamma')$ where $\Gamma'\subset \Gamma$ is a torsion-free subgroup of finite index, and in this case $\Gamma'N/N$ is a torsion-free subgroup of finite index of $\Gamma/N$ isomorphic to $\Gamma'$.
\end{demo}

We refer to \cite{Brown} and \cite{Borel-Serre} for facts about the (virtual) cohomological dimension and geometric dimension.

\subsection{Bredon cohomology}

The Bredon cohomological dimension $\underline{\cd}(\Gamma)$ is an algebraic counterpart to the proper geometric dimension $\underline{\gd}(\Gamma)$. We recall how $\underline{\cd}(\Gamma)$ is defined and a few of its properties.

Let $\Gamma$ be a discrete group and $\mathcal{F}$ be the family of subgroups of $\Gamma$. The \textit{orbit category} $\mathcal{O}_{\mathcal{F}}\Gamma$ is the category whose objects are left coset spaces $\Gamma/H$ with $H \in \mathcal{F}$ and where the morphisms are all $\Gamma$-equivariant maps between them. An $\mathcal{O}_{\mathcal{F}}\Gamma$\textit{-module} is a contravariant functor
\[ M: \mathcal{O}_{\mathcal{F}}\Gamma \rightarrow \mathbb{Z}\text{-mod} \]
to the category of $\mathbb{Z}$-modules. The \textit{category of $\mathcal{O}_{\mathcal{F}}\Gamma$-modules}, denoted by Mod-$\mathcal{O}_{\mathcal{F}}\Gamma$, has as objects all the $\mathcal{O}_{\mathcal{F}}\Gamma$-modules and all the natural transformations between them as morphisms. One can show that Mod-$\mathcal{O}_{\mathcal{F}}\Gamma$ is an abelian category and that we can construct projective resolutions on it. The \textit{Bredon cohomology of $\Gamma$} with coefficients in $M \in$ Mod-$\mathcal{O}_{\mathcal{F}}\Gamma$, denoted by $\Ho_{\mathcal{F}}^*(\Gamma,M)$, is by definition the cohomology associated to the cochain of complexes $\Hom_{\mathcal{O}_{\mathcal{F}}\Gamma}(C_*,M)$ where $C_* \rightarrow \underline{\mathbb{Z}}$ is a projective resolution of the functor $\underline{\mathbb{Z}}$ which maps all objects to $\mathbb{Z}$ and all morphisms to the identity map. If $X$ is a model for $\underline{E}\Gamma$, the augmented cellular chain complexes $C_*(X^H) \rightarrow \mathbb{Z}$ of the fixed points sets $X^H$ for $H \in \mathcal{F}$ form such a projective (even free) resolution $C_*(X^-) \rightarrow \underline{\mathbb{Z}}$. Thus we have
\[ \Ho_{\mathcal{F}}^n(\Gamma,M)=\Ho^n(\Hom_{\mathcal{O}_{\mathcal{F}}\Gamma}
(C_*(X^-),M)). \]
The \textit{Bredon cohomological dimension of $\Gamma$ for proper actions}, denoted by $\underline{\cd}(\Gamma)$ is defined as
\[ \underline{\cd}(\Gamma)=\sup\{n \in \mathbb{N}~|~ \exists M \in \text{Mod-}\mathcal{O}_{\mathcal{F}}\Gamma: \Ho_{\mathcal{F}}^n(\Gamma,M) \neq 0\}. \]
As we said above, this invariant can be viewed as an algebraic counterpart to $\underline{\gd}(\Gamma)$. Indeed Lück and Meintrup proved in \cite{Luck-Meintrup} the following theorem:

\begin{thm}[Lück-Meintrup]\label{Théorème Luck-Meintrup}
If $\Gamma$ is a discrete group with $\underline{\cd}(\Gamma) \geqslant 3$, then $\underline{\gd}(\Gamma)=\underline{\cd}(\Gamma)$.
\end{thm}

We explain now the strategy to prove that $\vcd(\Gamma)=\underline{\cd}(\Gamma)$, beginning with some material and definitions. Recall that if $G$ is the group of real points of a semisimple algebraic group and $\Gamma \subset G$ a lattice, then the Borel-bordification $X$ is a model for $\underline{E}\Gamma$. Note also that if $H$ is a finite subgroup of $\Gamma$, $\dim(X^H)=\dim(S^H)$. If we denote $\mathcal{F}_0$ the family of finite subgroups of $\Gamma$ containing properly the kernel of $\Gamma$, $X_{sing}$ the subspace of the Borel-bordification $X$ consisting of points whose stabilizer is stricly larger than the kernel of $\Gamma$, and
\[ \mathcal{S}=\{X^H~|~H \in \mathcal{F}_0 \text{ and } \nexists H' \in \mathcal{F}_0 \text{ with }X^H \subsetneq X^H \}, \]
then we have
\[ X_{sing}=\bigcup_{X^H \in \mathcal{S}} X^H. \]
Also every fixed point set $X^H \in \mathcal{S}$ is of the form $X^{\alpha}$ where $\alpha \in \Gamma$ is of finite order and non-central.

In general, computing $\underline{\cd}(\Gamma)$ is not an easy task. However, if $\Gamma$ admits a cocompact model $X$ for $\underline{E}\Gamma$, then there is a version of the formula \eqref{Expression vcd} for the Bredon cohomological dimension. In fact, from \cite[Th. 1.1]{Degrijse} we get that
\[ \underline{\cd}(\Gamma)=\max \{n \in \mathbb{N} ~|~ \exists K \in \mathcal{F}_0 ~ s.t. ~ \Ho_c^n(X^K,X_{sing}^K) \neq 0 \} \]
where $X^K$ is the fixed point set of $X$ under $K$ and $X_{sing}^K$ is the subcomplex of $X^K$ consisting of those cells that are fixed by a finite subgroup of $\Gamma$ that strictly contains $K$. 

 Using the above caracterisations of $\vcd(\Gamma)$ and $\underline{\cd}(\Gamma)$, one can show (see \cite[Prop. 3.3]{Aramayona})

\begin{prop}
Let $G$ be the group of real points of a semisimple algebraic group $\mathbb{G}$ of real rank at least two, $\Gamma \subset G$ a non-uniform lattice of $G$, $K\subset G$ a maximal compact subgroup and $S=G/K$ the associated Riemannian symmetric space. If 
\begin{enumerate}
\item $\dim (X^{\alpha}) \leqslant \vcd(\Gamma)$ for every $X^{\alpha} \in \mathcal{S}$, and
\item the homomorphism $\Ho^{\vcd(\Gamma)}_c(X) \rightarrow \Ho^{\vcd(\Gamma)}_c(X_{sing})$ is surjective
\end{enumerate}
then $\vcd(\Gamma)=\underline{\cd}(\Gamma)$.
\end{prop}

Note that in \cite{Aramayona} the authors assume that $\mathbb{G}$ is connected but this hypothesis is not needed as the Borel-Serre bordification is still a model for $\underline{E}\Gamma$ if $\mathbb{G}$ is not connected (see \cite[Th. 3.2]{Ji}).
As $\dim(X^{\alpha})=\dim(S^{\alpha})$ we have immediately the following lemma as a corollary of the previous proposition (see \cite[Cor. 3.4]{Aramayona})

\begin{lemma}\label{Lemme clé}
With the same notations as above, if $\dim S^{\alpha}<\vcd(\Gamma)$ for all $\alpha\in \Gamma$ of finite order and non central, then $\underline{\cd}(\Gamma)=\vcd(\Gamma)$.
\end{lemma}

This lemma will be the key argument to prove the Main Theorem. However, as it is the case in \cite{Aramayona}, in some cases we will need the following result (see \cite[Cor 3.7]{Aramayona})

\begin{lemma}\label{Lemme secours}
With the same notations as above, suppose that
\begin{enumerate}
\item $\dim(S^{\alpha}) \leqslant \vcd(\Gamma)$ for every non-central finite order element $\alpha \in \Gamma$,
\item $\dim(S^{\alpha} \cap S^{\beta}) \leqslant \vcd(\Gamma)-2$ for any distinct $S^{\alpha}, S^{\beta} \in \mathcal{S}$, and
\item for any finite set of non-central finite order elements $\alpha_1,\dots, \alpha_r$ with $S^{\alpha_i} \neq S^{\alpha_j}$ for $i\neq j$, $\dim(S^{\alpha_i})=\vcd(\Gamma)$, and such that $C_{\Gamma}(\alpha_i)$ is a cocompact lattice in $C_G(\alpha_i)$, there exists a rational flat $F$ in $\mathcal{S}$ that intersects $S^{\alpha_i}$ in exactly one point and is disjoint from $S^{\alpha_i}$ for $i \in \{2,\dots,n\}$.
\end{enumerate}
Then $\vcd(\Gamma)=\underline{\cd}(\Gamma)$.
\end{lemma}

We refer to \cite{Luck-Meintrup} and \cite{Degrijse} for other facts about Bredon cohomology.

\section{Complex simple Lie algebras}

In this section we prove the Main Theorem for all complex simple Lie algebras:

\begin{prop}\label{Prop groupe complexe}
Let $\mathfrak{g}$ be a complex simple Lie algebra, $G=\Aut(\mathfrak{g})$ its group of automorphisms and $S$ the associated Riemannian symmetric space. We assume that $\rk_{\mathbb{R}} G \geqslant 2$. Then
\[ \dim S^{\alpha} < \dim S- \rk_{\mathbb{R}} G \]
for every $\alpha \in G$ of finite order and non central. In particular
\[ \underline{\gd}(\Gamma)=\vcd(\Gamma) \]
for every lattice $\Gamma \subset G$.
\end{prop}

Recall that the adjoint group $G_{ad}$ is the identity component of $G=\Aut(\mathfrak{g})$. It agrees with the group of inner automorphisms, that is $G_{ad}=\Inn(\mathfrak{g})$. Note that $G_{ad}$ is centerfree has the same dimension and real rank as $G$, and their associated Riemannian symmetric spaces agree. The quotient $\Aut(\mathfrak{g})/\Inn(\mathfrak{g})$ is the group of outer automorphism $\Out(\mathfrak{g})$ and can be realized as a subgroup of $\Aut(\mathfrak{g})$. The group $\Aut(\mathfrak{g})$ is then the semi-direct product of $\Inn(\mathfrak{g})$ and $\Out(\mathfrak{g})$ (see \cite{Gundogane}). Recall also that if $A \in G_{ad}$, $S^{A}$ is the fixed point set of the inner automorphism $\Ad(A)$.

For further use, note that if $\rho \in \Aut(\mathfrak{g})$ is an involution, then it is induced by an involution on $G_{ad}$ that will be still denoted $\rho$. The group of its fixed points $G_{ad}^{\rho}$ has Lie algebra $\mathfrak{g}^{\rho}$ and the fixed point set $S^{\rho}$ is the associated Riemannian symmetric space. In particular, $\dim S^{\rho}=0$ if $G_{ad}^{\rho}$ is compact.

The proof of Proposition \ref{Prop groupe complexe} relies on the following lemmas:

\begin{lemma}\label{Méthode groupe complexe 2}
Let $\mathfrak{g}$ be a semisimple Lie algebra such that every element of $\Out(\mathfrak{g})$ has order at most 2. Let $G$ be the group of automorphisms of $\mathfrak{g}$ and let $S$ be the associated Riemannian symmetric space. If 
\begin{equation} \dim S^A < \dim S- \rk_{\mathbb{R}} G \text{ and } \dim S^{\rho} < \dim S- \rk_{\mathbb{R}}G
\label{Condition principale}
\end{equation}
for all $A \in G_{ad}$ non trivial of finite order, and for all involutions $\rho \in G$, then we also have
\[ \dim S^{\alpha} < \dim S- \rk_{\mathbb{R}} G \]
for every $\alpha \in G$ of finite order and non central.
\end{lemma}

\begin{proof}
Every element $\alpha \in \Aut(\mathfrak{g})$ is of the form $\Ad(A) \circ \rho$ where $A \in G_{ad}$ and $\rho \in \Out(\mathfrak{g}) \subset \Aut(\mathfrak{g})$. We know that $\rho$ is of order at most 2 by hypothesis. Then $\alpha^2=\Ad(A\rho(A))$ is an inner automorphism and we have the inclusion $S^{\alpha} \subset S^{\alpha^2}=S^{A\rho(A)}$. So, if $A\rho(A)$ is not central in $G_{ad}$, then we have
\[ \dim S^{\alpha} \leqslant \dim S^{A\rho(A)} < \dim S- \rk_{\mathbb{R}} G. \]
Note now that if $A\rho(A)$ is central, then it is actually the identity because $G_{ad}$ is centerfree. This means that $\alpha^2=Id$. In other words, $\alpha$ is an involution, and we again have
\[ \dim S^{\alpha} < \dim S- \rk_{\mathbb{R}} G \]
by assumption. We have proved the claim.
\end{proof}

To check the first part of \eqref{Condition principale} we will use the following:

\begin{lemma}\label{Méthode groupe complexe}
Let $G$ be the group of complex points of a semisimple connected algebraic group and $K \subset G$ a maximal compact subgroup. Suppose that there exists a group $H$ isogenous to a subgroup $H'$ of $K$ such that $K/H'$ is an irreducible symmetric space, $\rk K=\rk H$, $\dim H < \dim K - \rk_{\mathbb{R}} G$, and satisfying 
\[ \dim C_H(A) < \dim H - \rk_{\mathbb{R}} G \]
for all $A \in H$ of finite order and non central.
Then we have 
\[ \dim S^A < \dim S- \rk_{\mathbb{R}} G \]
for every $A \in G$ of finite order and non central.
\end{lemma}

\begin{proof}
As all maximal compact subgroups are conjugated, we can conjugate such an $A \in G$ into $K$. Since $K$ is connected,  $A$ is then contained in a maximal torus. Since all maximal tori are conjugated, we can conjugate $A$ into any one of them. Since the subgroup $H'$ has the same rank as $K$, a maximal torus in $H'$ is also maximal in $K$. We can then assume, up to replacing $A$ by a conjugate element, that $A \in H'$. 

Taking now into account that $G$ is the group of complex points of a reductive algebraic group, $C_G(A)$ is the complexification of its maximal compact subgroup $C_K(A)$ and then its dimension is twice that of $C_K(A)$. As a result we get
\[ \dim S^A=\dim C_G(A)-\dim C_K(A)=\dim C_K(A) \]
because $S^A \simeq C_G(A)/C_K(A)$ as seen in section 2.5.

Similarly we have
\[ \dim S=\dim(G/K)=\dim K.\]

In particular, the claim follows once we show that
\[ \dim C_K(A) < \dim K -\rk_{\mathbb{R}} G. \]

Now, as $H$ and $H'$ are isogenous, assume for simplicity that $H=H'/F$ with $F$ a finite normal subgroup of $H'$. We denote by $\bar{A}$ the class of $A$ in $H'/F$. As $F$ is finite, we have
\[ \dim C_H(\bar{A})=\dim C_{H'}(A) \]
and in particular, $A$ is central in $H'$ if and only if $\bar{A}$ is central in $H$.

Suppose for a moment that $A$ is non central in $H'$, then we write
\[ \dim C_K(A) \leqslant \dim C_{H'}(A) + \dim K - \dim H= \dim C_H(\bar{A})+\dim K-\dim H \]
and by assumption we have
\[ \dim C_H(\bar{A}) < \dim H - \rk_{\mathbb{R}} G \]
and finally
\[ \dim C_K(A) < \dim K - \rk_{\mathbb{R}} G. \]

It remains to treat the case that $A$ is central in $H'$ but not in $K$, that is $H' \subset C_K(A) \subsetneq K$. Since the symmetric space $K/H'$ is irreducible, it follows that the identity component of $H'$ is a maximal connected Lie group of $K$, so $\dim C_K(A)=\dim H'=\dim H$, and we have
\[ \dim H  < \dim K - \rk_{\mathbb{R}} G\]
by assumption.
\end{proof}

In the course of the proof of Proposition \ref{Prop groupe complexe}, the subgroups $H$ will all be classical groups of the forms $\SO(n)$ or $\SU(n)$ and we will need the following bounds for the dimension of centralizers in those groups (see \cite{Aramayona} section 5)

1. Let $A \in \SO(n)$ ($n\geqslant 3$) of finite order and non central, then
\begin{equation}
\label{Equation 1 complexe}
 \dim C_{\SO(n)}(A) \leqslant \frac{(n-1)(n-2)}{2}.
\end{equation}

2. Let $A  \in \SU(n)$ ($n \geqslant 2$) of finite order and non central, then
\begin{equation}
\label{Equation 2 complexe}
 C_{\SU(n)}(A) \leqslant (n-1)^2.
\end{equation}

For more simplicity we will sometimes consider $H$ as a subgroup of $K$ and denote $K/H$ the symmetric space $K/H'$.
For the convenience of the reader, we summarize in the following table the informations we need to prove Proposition \ref{Prop groupe complexe} for exceptional Lie algebras.

\begin{center}
\begin{tabular}{|*{7}{c|}}
\hline
$G_{ad}$ & $K$ & $H$ & $\dim K$ & $\dim H$ & $\rk H=\rk K$ & $\rk_{\mathbb{R}} G_{ad}$ \\
\hline
$G_2^{\mathbb{C}}$ & $G_2$ & $\SO(4)$ & 14 & 6 & 2 & 2 \\
\hline
$F_4^{\mathbb{C}}$ & $F_4$ & $\SO(9)$ & 52 & 36 & 4 & 4 \\
\hline
$E_6^{\mathbb{C}}$ & $E_6$ &  $\U(1)\times \SO(10)$ & 78 & 46 & 6 & 6 \\
\hline
$E_7^{\mathbb{C}}$ & $E_7$ &  $\SU(8)$ & 133 & 63 & 7 & 7 \\
\hline
$E_8^{\mathbb{C}}$ & $E_8$ &  $\SO(16)$ & 248 & 120 & 8 & 8 \\
\hline
\end{tabular}
\captionof{table}{Exceptional complex simple centerless Lie groups $G_{ad}$, maximal compact subgroups $K$, classical subgroups $H$, dimensions and ranks.}
\end{center}

We are now ready to launch the proof of Proposition \ref{Prop groupe complexe}.

\begin{proof}[Proof of Proposition \ref{Prop groupe complexe}]

The second claim follows from Lemma \ref{Lemme clé} because we have
\[ \vcd(\Gamma) \geqslant \dim S-\rk_{\mathbb{R}} G \]
for every lattice $\Gamma \subset G$  by Theorem \ref{Theoreme Borel-Serre}. So it suffices to prove the first claim.

Recall that every complex simple Lie algebra is isomorphic to either one of the classical algebras $\mathfrak{sl}(n,\mathbb{C})$, $\mathfrak{so}(n,\mathbb{C})$ and $\mathfrak{sp}(2n,\mathbb{C})$ (with conditions on $n$ to ensure simplicity), or to one of the 5 exceptional ones: $\mathfrak{g}_2^{\mathbb{C}}$, $\mathfrak{f}_4^{\mathbb{C}}$, $\mathfrak{e}_6^{\mathbb{C}}$, $\mathfrak{e}_7^{\mathbb{C}}$ and $\mathfrak{e}_8^{\mathbb{C}}$. 

To prove Proposition \ref{Prop groupe complexe} we will consider all these cases individually.

\vspace{0.5cm}

\textbf{Classical complex simple Lie algebras}:

Let $\mathfrak{g}$ be a classical complex simple Lie algebra and $\Gamma \subset G=\Aut(\mathfrak{g})$ a lattice. From a brief inspection of the table in section 2.3 we obtain that, unless $\mathfrak{g}=\mathfrak{so}(8,\mathbb{C})$, every outer automorphism of $\mathfrak{g}$ has order 2. We will assume that $\mathfrak{g} \neq \mathfrak{so}(8,\mathbb{C})$ for a while, treating this case later.

To begin with note that we get from parts 6.1, 6.2 and 6.3 of \cite{Aramayona} that
\[ \dim S^A < \dim S - \rk_{\mathbb{R}} G \]
for all $A \in G_{ad}$ non trivial and of finite order. In other words, the first part of condition \eqref{Condition principale} in Lemma \ref{Méthode groupe complexe 2} holds. To check the second part we make use of the classification of local symmetric spaces in \cite{Berger}. For instance if $\mathfrak{g}=\mathfrak{sl}(n,\mathbb{C})$ with $n \geqslant 3$ because of our assumption on the rank, and if $\rho \in G$ is an involution, then the Lie algebra $\mathfrak{g}^{\rho}$ is isomorphic to either a Lie algebra whose adjoint group is compact or to one of the follows: $\mathfrak{so}(n,\mathbb{C})$, $\mathfrak{s(gl(}k,\mathbb{C}) \oplus \mathfrak{gl}(n-k,\mathbb{C}))$, $\mathfrak{sp}(n,\mathbb{C})$, $\mathfrak{sl}(n,\mathbb{R})$,  $\mathfrak{su}(p,n-p)$ and $\mathfrak{sl}(\frac{n}{2},\mathbb{H})$, where $\mathfrak{sp}(n,\mathbb{C})$ and $\mathfrak{sl}(\frac{n}{2},\mathbb{H})$ only appear if $n$ is even. The associated Riemannian symmetric space $S^{\rho}$ is obtained by taking the quotient of the adjoint group by a maximal compact subgroup, for example in the case of $\mathfrak{so}(n,\mathbb{C})$ it is $\PSO(n,\mathbb{C})/\PSO_n$. The Lie algebra $\mathfrak{g}^{\rho}$ for which $\dim S^{\rho}$ is maximal is $\mathfrak{s(gl(}1,\mathbb{C}) \oplus \mathfrak{gl}(n-1,\mathbb{C}))$ for $n\neq 4$ (where $\dim S^{\rho}=(n-1)^2$), and $\mathfrak{sp}(4,\mathbb{C})$ for $n=4$ (where $\dim S^{\rho}=10$). In all these cases we have
\[ \dim S^{\rho} < (n^2-1)-(n-1)= \dim S - \rk_{\mathbb{R}} G. \]
We get then by Lemma \ref{Méthode groupe complexe 2} that the first claim of Proposition \ref{Prop groupe complexe} holds. The cases of $\mathfrak{sp}(2n,\mathbb{C})$ and $\mathfrak{so}(n,\mathbb{C})$ for $n\neq 8$ are similar, we leave the details to the reader.

Now we treat the case of the Lie algebra  $\mathfrak{so}(8,\mathbb{C})$. Its group of complex outer automorphisms is isomorphic to the symmetric group $S_3$ and contains an order 3 element $\tau$ called triality. (See section 1.14 in \cite{Yokota} for an interpretation of triality in terms of octonions.) The group $\Out(\mathfrak{so}(8,\mathbb{C}))$ of real outer automorphisms is then isomorphic to $S_3 \times \mathbb{Z}/2\mathbb{Z}$ where the second factor corresponds to complex conjugation. Consequently the only order 3 outer automorphisms are $\tau$ and $\tau^{-1}$ and those of order 6 are their compositions with complex conjugation. If $\rho \in \Aut(\mathfrak{so}(8,\mathbb{C}))$ is of order 6 and $\alpha=\Ad(A) \circ \rho$, then $\alpha^3$ is the composition of an inner automorphism and $\rho^3$. As $\rho^3$ is of order 2 and we have the inclusions $S^{\alpha} \subset S^{\alpha^3}$, we can consider $\alpha^3$ instead of $\alpha$ and we just have to treat the cases when $\rho$ is of order 2 or 3.

If $\rho$ is of order 2 we apply the same method than for other classical simple Lie algebras, using the classification of local symmetric spaces. It remains to treat the case when $\rho=\tau$ the triality automorphism (or its inverse). In this case $\alpha=\Ad(A) \circ \tau$ is a complex automorphism. Proceeding like in the proof of Lemma \ref{Méthode groupe complexe 2}, $\alpha^3$ is an inner automorphism and the result follows if it is non trivial. If $\alpha^3=1$ then $\alpha$ belongs to the set $\Aut_{\mathbb{C}}^3(\mathfrak{so}(8,\mathbb{C}))$ of complex automorphisms of order 3. A result of Gray and Wolf (see \cite[Thm. 5.5]{Gray-Wolf}) says that if $\sim_i$ is the equivalence relation of conjugation by an inner automorphism in $\Aut_{\mathbb{C}}^3(\mathfrak{so}(8,\mathbb{C}))$, then $\Aut_{\mathbb{C}}^3(\mathfrak{so}(8,\mathbb{C}))/\sim_i$ contains, besides the classes of inner automorphisms, four other classes: those of $\tau$ and $\tau^{-1}$ and two others represented by order 3 automorphisms $\tau'$ and $\tau'^{-1}$. The Lie algebra of the fixed point set of triality is the exceptional Lie algebra $\mathfrak{g}_2^{\mathbb{C}}$, and those of the fixed point set of $\tau'$ is isomorphic to the Lie algebra $\mathfrak{sl}(3,\mathbb{C})$. In both cases we have
\[ \dim S^{\alpha} < \dim S - \rk_{\mathbb{R}} G \]
and Proposition \ref{Prop groupe complexe} holds for $\mathfrak{g}=\mathfrak{so}(8,\mathbb{C})$.

\vspace{0.5cm}

\textbf{Lie algebra $\mathfrak{g}_2^{\mathbb{C}}$:}

Here we consider the simple exceptional Lie algebra $\mathfrak{g}=\mathfrak{g}_2^{\mathbb{C}}$. Its outer automorphism group is of order 2 and its adjoint group is $G_{ad}=G_2^{\mathbb{C}}$, a connected algebraic group of real rank 2 and complex dimension 14. The compact group $G_2$ is a maximal compact subgroup of $G_2^{\mathbb{C}}$.

The group $G_2$ contains a subgroup $H$ isomorphic to $\SO(4)$, fixed by an involution of $G_2$ which extends to a conjugation of $G_2^{\mathbb{C}}$, giving the split real form $G_{2(2)}$ (see section 1.10 of \cite{Yokota} for an explicit description). Then $G_2/\SO(4)$ is an irreducible symmetric space, $\rk(G_2)=\rk(\SO(4))=2$ and 
  \[ \dim H+\rk_{\mathbb{R}} G_{ad}=6+2=8 < 14=\dim K. \]
Moreover if $A \in H$ is of finite order and non central in $H \simeq \SO(4)$ we have by inequality \eqref{Equation 1 complexe}
\[ \dim C_H(A)+\rk_{\mathbb{R}}G_{ad} \leqslant 3+2=5 < 6=\dim H. \]

So by Lemma \ref{Méthode groupe complexe}, the first part of \eqref{Condition principale} in Lemma \ref{Méthode groupe complexe 2} holds. 

To check the second part we have to list the local symmetric spaces associated to an involution $\rho \in \Aut(\mathfrak{g}_2^{\mathbb{C}})$. By the classification of Berger in \cite{Berger}, the only non compact cases are when $\mathfrak{g}^{\rho}$ is isomorphic to $\mathfrak{sl}(2,\mathbb{C}) \oplus \mathfrak{sl}(2,\mathbb{C})$ or $\mathfrak{g}_{2(2)}$. The associated Riemannian symmetric spaces are $S^{\rho}=(\PSL(2,\mathbb{C}) \times \PSL(2,\mathbb{C})) / (\PSU_2 \times \PSU_2)$ and $G_{2(2)}/\SO(4)$ and we have in both cases
\[ \dim S^{\rho} < \dim S- \rk_{\mathbb{R}} G. \]

So by Lemma \ref{Méthode groupe complexe 2}, Proposition \ref{Prop groupe complexe} holds for $G=\Aut(\mathfrak{g}_2^{\mathbb{C}})$.

\vspace{0.5 cm}

\textbf{Lie algebra $\mathfrak{f}_4^{\mathbb{C}}$:}

We proceed like previously for the simple Lie algebra $\mathfrak{g}=\mathfrak{f}_4^{\mathbb{C}}$, with $|\Out(\mathfrak{f}_4^{\mathbb{C}})|=2$, $G_{ad}=F_4^{\mathbb{C}}$ of maximal compact subgroup $K=F_4$, $\rk_{\mathbb{R}}G_{ad}=4$ and $\dim K=52$. We know that there exists a subgroup $H \subset K$ isogenous to $\SO(9)$, with $\rk H=\rk K=4$, and such that $K/H$ is an irreducible symmetric space. In addition to that
 \[  \dim H+\rk_{\mathbb{R}} G_{ad}=36+4=40 < 52=\dim K \]
and if $A \in H$ is of finite order and non central in $H$, we have 
\[ \dim C_H(A)+\rk_{\mathbb{R}}G_{ad} \leqslant 28+4=32 <36= \dim H \]
by inequality \eqref{Equation 1 complexe}.

So by Lemma \ref{Méthode groupe complexe}, the first part of \eqref{Condition principale} in Lemma \ref{Méthode groupe complexe 2} holds. 

Then by the classification of local symmetric spaces, the ones we have to study is those when $\mathfrak{g}^{\rho}$ is isomorphic to $\mathfrak{sp}(6,\mathbb{C}) \oplus \mathfrak{sp}(2,\mathbb{C})$, $\mathfrak{so}(9,\mathbb{C})$,  $\mathfrak{f}_{4(4)}$ or $\mathfrak{f}_{4(-20)}$, in all cases
\[ \dim S^{\rho} < \dim S- \rk_{\mathbb{R}} G. \]

So by Lemma \ref{Méthode groupe complexe 2}, Proposition \ref{Prop groupe complexe} holds for $G=\Aut(\mathfrak{f}_4^{\mathbb{C}})$.

\vspace{0.5 cm}

\textbf{Lie algebra $\mathfrak{e}_6^{\mathbb{C}}$:}

For the algebra $\mathfrak{g}=\mathfrak{e}_6^{\mathbb{C}}$, the outer automorphism group is a product of two groups of order 2, $G_{ad}=E_6^{\mathbb{C}}$ of maximal compact subgroup $K=E_6$, $\rk_{\mathbb{R}}G_{ad}=6$ and $\dim K=78$. We know there exists $H \subset K$ isogenous to $\U(1)\times \SO(10)$, with $\rk H=\rk K=6$, $K/H$ is an irreducible symmetric space and we have the following
 \[ \dim H+\rk_{\mathbb{R}} G_{ad}=46+6=52 < 78=\dim K \]
and if $A \in H$ is of finite order and non central in $H$, by inequality \eqref{Equation 1 complexe} we have 
\[ \dim C_H(A)+\rk_{\mathbb{R}}G_{ad} \leqslant 1+36+6=43 <46= \dim H. \]

So by Lemma \ref{Méthode groupe complexe}, the first part of \eqref{Condition principale} in Lemma \ref{Méthode groupe complexe 2} holds. 

Then by the classification of local symmetric spaces, the ones we have to study is those when $\mathfrak{g}^{\rho}$ is isomorphic to $\mathfrak{sp}(8,\mathbb{C})$, $\mathfrak{sl}(6,\mathbb{C}) \oplus \mathfrak{sl}(2,\mathbb{C})$, $\mathfrak{so}(10,\mathbb{C})\oplus \mathfrak{so}(2,\mathbb{C})$, $\mathfrak{f}_4^{\mathbb{C}}$, $\mathfrak{e}_{6(6)}$, $\mathfrak{e}_{6(2)}$, $\mathfrak{e}_{6(-14)}$ or $\mathfrak{e}_{6(-26)}$, in all cases
\[ \dim S^{\rho} < \dim S- \rk_{\mathbb{R}} G. \]

So by Lemma \ref{Méthode groupe complexe 2}, Proposition \ref{Prop groupe complexe} holds for $G=\Aut(\mathfrak{e}_6^{\mathbb{C}})$.

\vspace{0.5 cm}

\textbf{Lie algebra $\mathfrak{e}_7^{\mathbb{C}}$:}

Consider now the simple Lie algebra $\mathfrak{g}=\mathfrak{e}_7^{\mathbb{C}}$, of order 2 outer automorphism group, and of adjoint group $G_{ad}=E_7^{\mathbb{C}}$, whose compact maximal subgroup is $K=E_7$, $\rk_{\mathbb{R}}G_{ad}=7$ and $\dim K=133$. We know there exists $H \subset K$ isogenous to $\SU(8)$ with $\rk H=\rk K=7$ and $K/H$ is an irreducible symmetric space. We have the inequality
\[ \dim H+\rk_{\mathbb{R}} G_{ad}=63+7=70 < 133=\dim K \]
and if $A \in H$ is of finite order and non central in $H$, by inequality \eqref{Equation 2 complexe} we have
\[ \dim C_H(A)+\rk_{\mathbb{R}}G_{ad} \leqslant 49+7 <63= \dim H. \]

So by Lemma \ref{Méthode groupe complexe}, the first part of \eqref{Condition principale} in Lemma \ref{Méthode groupe complexe 2} holds. 

Then by the classification of local symmetric spaces, the ones we have to study is those when $\mathfrak{g}^{\rho}$ is isomorphic to $\mathfrak{sl}(8,\mathbb{C})$, $\mathfrak{so}(12,\mathbb{C}) \oplus \mathfrak{sp}(2,\mathbb{C})$, $\mathfrak{e}_6^{\mathbb{C}} \oplus \mathfrak{so}(2,\mathbb{C})$, $\mathfrak{e}_{7(7)}$, $\mathfrak{e}_{7(5)}$ or $\mathfrak{e}_{7(-25)}$, in all cases
\[ \dim S^{\rho} < \dim S- \rk_{\mathbb{R}} G. \]

So by Lemma \ref{Méthode groupe complexe 2}, Proposition \ref{Prop groupe complexe} holds for $G=\Aut(\mathfrak{e}_7^{\mathbb{C}})$.

\vspace{0.5 cm}

\textbf{Lie algebra $\mathfrak{e}_8^{\mathbb{C}}$:}

The last exceptional Lie algebra is $\mathfrak{g}=\mathfrak{e}_8^{\mathbb{C}}$, again its outer automorphism group is of order 2, its adjoint group is $G_{ad}=E_8^{\mathbb{C}}$, of maximal compact subgroup $K=E_8$, $\rk_{\mathbb{R}}G_{ad}=8$ and $\dim K=248$. We know there exists $H \subset K$ isogenous to $\SO(16)$, with $\rk H=\rk K=8$ and $K/H$ is an irreducible symmetric space. We have also the inequality
\[ \dim H+\rk_{\mathbb{R}} G_{ad}=120+8=128 < 248=\dim K \]
and if $A \in H$ is of finite order and non central in $H$, by inequality \eqref{Equation 1 complexe}  we have
\[ \dim C_H(A)+\rk_{\mathbb{R}}G_{ad} \leqslant 105+8 <120= \dim H. \]

So by Lemma \ref{Méthode groupe complexe}, the first part of \eqref{Condition principale} in Lemma \ref{Méthode groupe complexe 2} holds. 

Then by the classification of local symmetric spaces, the ones we have to study is those when $\mathfrak{g}^{\rho}$ is isomorphic to $\mathfrak{so}(16,\mathbb{C})$, $\mathfrak{e}_7^{\mathbb{C}} \oplus \mathfrak{sp}(2,\mathbb{C})$, $\mathfrak{e}_{8(8)}$ or $\mathfrak{e}_{8(-24)}$, in all cases
\[ \dim S^{\rho} < \dim S- \rk_{\mathbb{R}} G. \]

So by Lemma \ref{Méthode groupe complexe 2}, Proposition \ref{Prop groupe complexe} holds for $G=\Aut(\mathfrak{e}_8^{\mathbb{C}})$ and it concluded its proof.
\end{proof}

\section{Real simple Lie algebras}

We will in this section extend the previous proposition to the real simple Lie algebras. They are the real forms of the complex ones studied in the previous section. The ideas of the proof are similar to those of the complex case, although we face some additional difficulties. Maybe the reader can skip this section in a first reading.

\begin{prop}\label{Proposition groupe réel}
Let $\mathfrak{g}$ be a real simple Lie algebra, $G=\Aut(\mathfrak{g})$ its group of automorphisms and $S$ the associated Riemannian symmetric space. Then
\[ \underline{\gd}(\Gamma)=\vcd(\Gamma) \]
for every lattice $\Gamma \subset G$. Moreover
\[ \dim S^{\alpha} \leqslant \dim S-\rk_{\mathbb{R}}G \]
for every $\alpha \in G$ of finite order and non central.
\end{prop}

We will again use Lemma \ref{Méthode groupe complexe 2}, but in the case of exceptional real simple Lie algebras, we cannot use Lemma \ref{Méthode groupe complexe} to establish inequalities of the form
\[ \dim S^A < \dim S - \rk_{\mathbb{R}} G \]
for $A$ in the adjoint group $G_{ad}$ of $\mathfrak{g}$. The difficulty is that the dimension of $G_{ad}$ is not anymore twice that of a maximal compact subgroup. To some extent, we will bypass this problem using the following lemma:

\begin{lemma}\label{Méthode groupe réel}
Let $G$ be a connected Lie group which is the group of real points of a semisimple algebraic group defined over $\mathbb{R}$, and $K \subset G$ a maximal compact subgroup.  Suppose there exists a subgroup $\bar{G} \subset G$ such that $G/\bar{G}$ is an irreducible symmetric space and whose compact maximal subgroup $\bar{K} \subset K$ has the same rank as $K$. Let $S=G/K$ and $\bar{S}=\bar{G}/\bar{K}$ be the associated Riemannian symmetric spaces. If we have  
\[ \dim \bar{S} < \dim S - \rk_{\mathbb{R}} G, \]
and
\[ \dim \bar{S}^A < \dim \bar{S} - \rk_{\mathbb{R}} G \]
for every $A \in \bar{K}$ of finite order non central, then we also have
\[ \dim S^A < \dim S - \rk_{\mathbb{R}} G \]
for every $A \in G$ of finite order and non central.
\end{lemma}

\begin{proof}
As in the proof of Lemma \ref{Méthode groupe complexe} we can conjugate such an $A$ into a maximal torus in $\bar{K}$. 

If $A$ is central in $\bar{G}$: as $G/\bar{G}$ is irreducible, we have that the identity component of $\bar{G}$ is a maximal connected Lie subgroup of $G$. It follows thus from $\bar{G}\subset C_G(A) \subsetneq G$ that the Riemannian symmetric spaces of $C_G(A)$ and $\bar{G}$ are the same, that is $S^A=\bar{S}$. Then the result follows from the assumption
\[\dim \bar{S} < \dim S- \rk_{\mathbb{R}} G. \]
Suppose now that $A$ is not central in $\bar{G}$, then we have
\[ \dim S - \dim S^A \geqslant \dim \bar{S}-\dim \bar{S}^A \]
as $\bar{S}^A=\bar{S} \cap S^A$. Then the result follows because
\[ \dim \bar{S} - \dim \bar{S}^A > \rk_{\mathbb{R}} G \]
by assumption.
\end{proof}

In all cases if interest, the group $\bar{G}$ will be a classical group. We will use the following inequalities  to majorate $\dim \bar{S}^A$ (see \cite{Aramayona} sections 6.5, 6.6 and 6.8):

1. Let $A \in \SU(p,q)$ ($p \leqslant q$, $ p+q \geqslant 3$) of finite order non central, and $S=\SU(p,q)/\mathrm{S}(\U(p)\times \U(q))$ the associated symmetric space, then
\begin{equation}
\label{Equation 1 réelle}
 \dim S^A \leqslant 2p(q-1).
\end{equation}

2. Let $A \in \Sp(p,q)$ ($p \leqslant q$, $p+q \geqslant 3$) of finite order non central, and $S=\Sp(p,q)/(\Sp(p)\times \Sp(q))$ the associated symmetric space, then
\begin{equation}
\label{Equation 2 réelle}
\dim S^A \leqslant 4p(q-1).
\end{equation}

3. Let $A \in \SO^*(2n)$ ($n \geqslant 2$) of finite order non central, et $S=\SO^*(2n)/\U(n)$ the associated symmetric space, then
\begin{equation}
\label{Equation 3 réelle}
 \dim S^A \leqslant n^2-n-2(n-1).
\end{equation}

The tables below list exceptional real simple Lie groups, the subgroups $\bar{G}$ we will use and the informations we need to know for the proof of Proposition \ref{Proposition groupe réel}. Note that for more simplicity, the compact maximal subgroups $K$ are given up to isogeny.

\begin{center}
\begin{tabular}{|*{4}{c|}}
\hline
$G_{ad}$ & $K$ & $\bar{G}$ & $\bar{K}$  \\
\hline
$E_{6(6)} $ & $\Sp(4)$ & $\Sp(2,2)$ & $\Sp(2)\times \Sp(2)$ \\
\hline
$E_{6(2)} $ & $\SU(6) \times \SU(2)$ & $\SO^*(10) \times \SO(2)$ & $\U(5) \times \SO(2)$ \\
\hline
$E_{6(-14)} $ & $\SO(10) \times \SO(2)$ & $\SO^*(10) \times \SO(2)$ & $\U(5) \times \SO(2)$ \\
\hline
$E_{6(-26)} $ & $F_4$ & $\Sp(1,3)$ & $\Sp(1) \times \Sp(3)$ \\
\hline
$E_{7(7)} $ & $\SU(8)$ & $E_{6(2)}\times \SO(2)$ & $\SU(6) \times \SU(2) \times \SO(2)$\\
\hline
$E_{7(-5)} $ & $\SO(12)\times \SU(2)$ & $\SU(4,4)$ & $\mathrm{S}(\U(4) \times \U(4))$\\
\hline
$E_{7(-25)} $ & $E_6 \times \SO(2)$ & $\SU(2,6)$ & $\mathrm{S}(\U(2) \times \U(6))$ \\
\hline
$E_{8(8)} $ & $\SO(16)$ & $\SO^*(16)$ & $\U(8)$\\
\hline
$E_{8(-24)} $ & $E_7 \times \SU(2)$ & $\SO^*(16)$ & $\U(8)$\\
\hline
$G_{2(2)}$ & $\SO(4)$ & $\SL(2,\mathbb{R}) \times \SL(2,\mathbb{R})$ & $\SO(2) \times \SO(2)$ \\
\hline
$F_{4(4)}$ & $\Sp(3) \times \Sp(1)$ & $\SO(4,5)$ & $\mathrm{S}(\mathrm{O}(4) \times \mathrm{O}(5))$ \\
\hline
\end{tabular}
\captionof{table}{Real exceptional simple centerless Lie groups $G_{ad}$, certain classical subgroups $\bar{G} \subset G_{ad}$ and the respective maximal compact subgroups}
\end{center}

\begin{center}
\begin{tabular}{|*{6}{c|}}
\hline
$G_{ad}$ &  $\dim S$ & $\dim \bar{S}$ &  $\rk \bar{K}=\rk K$ & $\rk_{\mathbb{R}} G_{ad}$ \\
\hline
$E_{6(6)} $   &  42 & 16 & 4 & 6 \\
\hline
$E_{6(2)} $  & 40 & 20 & 6 & 4\\
\hline
$E_{6(-14)} $  & 32 & 20 & 6 & 2\\
\hline
$E_{6(-26)} $   & 26 & 12 & 4 & 2 \\
\hline
$E_{7(7)} $  & 70 & 40 & 7 & 7 \\
\hline
$E_{7(-5)} $  & 64 & 32 & 7 & 4 \\
\hline
$E_{7(-25)} $  & 54 & 24 & 7 & 3\\
\hline
$E_{8(8)} $   & 128 & 56 & 8 & 8\\
\hline
$E_{8(-24)} $  & 112 & 56 & 8 & 4\\
\hline
$G_{2(2)}$ & 8 & 4 & 2 & 2 \\
\hline
$F_{4(4)}$  & 28 & 20 & 4 & 4 \\
\hline
\end{tabular}
\captionof{table}{With the same notations as in Table 11, dimensions of the Riemannian symmetric spaces associated to $G_{ad}$ and $\bar{G}$, together with the ranks of $K$, $\bar{K}$ and $G_{ad}$}
\end{center}

We are now ready to prove Proposition \ref{Proposition groupe réel}. 

\begin{proof}[Proof of Proposition \ref{Proposition groupe réel}]

Recall that the first claim holds when the adjoint group has real rank 1 by Proposition \ref{Prop rang 1}, that is when $\mathfrak{g}$ is isomorphic to $\mathfrak{sl}(2,\mathbb{R})\cong \mathfrak{sp}(2,\mathbb{R})$, $\mathfrak{sl}(2,\mathbb{H})$, $\mathfrak{so}^*(4)$, $\mathfrak{so}^*(6)$, $\mathfrak{so}(1,n)$, $\mathfrak{su}(1,n)$, $\mathfrak{sp}(1,n)$ and $\mathfrak{f}_{4(-20)}$. The second claim is also true because if $\alpha \in \Aut(\mathfrak{g})$ is of finite order and non central, then $S^{\alpha}$ is a strict submanifold of $S$ so we have
\[ \dim S^{\alpha} \leqslant \dim S - 1. \]

We suppose from now on that $\rk_{\mathbb{R}} G \geqslant 2$. By inspection of the table of outer automorphisms in section 2.3, we see that every outer automorphism of  $\mathfrak{g}$ has order 2, except if $\mathfrak{g}=\mathfrak{so}(p,p)$ with $p \geqslant 4$ even.

As in the proof of Proposition \ref{Prop groupe complexe}
We will again do a case-by-base analysis. 

\vspace{0.5cm}

\textbf{Classical real simple Lie algebras other that $\mathfrak{sl}(n,\mathbb{R})$ and $\mathfrak{so}(p,q)$}: 

We start dealing with the classical Lie algebras $\mathfrak{su}(p,q)$, $\mathfrak{sl}(n,\mathbb{H})$, $\mathfrak{sp}(2n,\mathbb{R})$ and $\mathfrak{so}^*(2n)$. Note that we rule out $\mathfrak{su}(2,2) \cong \mathfrak{so}(2,4)$ and $\mathfrak{sp}(4,\mathbb{R}) \cong \mathfrak{so}(2,3)$. 

We use again Lemma \ref{Méthode groupe complexe 2}. We want then to establish 
\[ \dim S^A < \dim S- \rk_{\mathbb{R}} G \text{ and } \dim S^{\rho} < \dim S- \rk_{\mathbb{R}} G\]
for every $A$ in the adjoint group $G_{ad}$ of finite order and non central and for every involution $\rho \in G=\Aut(\mathfrak{g})$. The first condition holds by the computations in sections 6.4 to 6.7 of \cite{Aramayona}. Using the classification of local symmetric spaces, we can check the second condition as we did in the complex case. For instance, if $\mathfrak{g}=\mathfrak{sp}(2n,\mathbb{R})$ with $n \geqslant 3$, then $\mathfrak{g}^{\rho}$ is either compact or isomorphic to one of the following: $\mathfrak{sp}(2k,\mathbb{R}) \oplus \mathfrak{sp}(2(n-k),\mathbb{R})$, $\mathfrak{u}(k,n-k)$, $\mathfrak{gl}(n,\mathbb{R})$ or $\mathfrak{sp}(n,\mathbb{C})$, the last case only appearing if $n$ is even. The Lie algebra $\mathfrak{g}^{\rho}$ for which $\dim S^{\rho}$ is maximal is $\mathfrak{sp}(2(n-1),\mathbb{R}) \oplus \mathfrak{sp}(2,\mathbb{R})$, for which we have
\[ \dim S^{\rho}=n^2-n+2 < n^2 = \dim S - \rk_{\mathbb{R}} G. \]
Hence by Lemma \ref{Méthode groupe complexe 2} and Lemma \ref{Lemme clé}, Proposition \ref{Proposition groupe réel} holds for $G=\Aut(\mathfrak{sp}(2n,\mathbb{R}))$. The cases of $\mathfrak{su}(p,q)$, $\mathfrak{sl}(n,\mathbb{H})$ and $\mathfrak{so}^*(2n)$ are similar.

\vspace{0.5cm}

\textbf{Lie algebras $\mathfrak{sl}(n,\mathbb{R})$ and $\mathfrak{so}(p,q)$}
The remaining classical cases are $\mathfrak{sl}(n,\mathbb{R})$ and $\mathfrak{so}(p,q)$. If $p \geqslant 6$ is even, then $\Out(\mathfrak{so}(p,p))$ is isomorphic to $D_4$ so every outer automorphism  $\rho$ has order 2 or 4. The only case where we have order 3 outer automorphisms is $\mathfrak{so}(4,4)$, as $\Out(\mathfrak{so}(4,4))$ is isomorphic to $S_4$.

As already noted in \cite{Aramayona}, where the argument for lattices in $\SL(n,\mathbb{R})$ and $\SO(p,q)$ was more involved than in the other cases, we face the problem that there exists $\alpha \in \Aut(\mathfrak{g})$ such that
\[ \dim S^{\alpha}=\dim S-\rk_{\mathbb{R}} G. \]
Our next goal is to characterize the automorphisms $\alpha$ for which this happens.

\begin{lemma}\label{Lemme sl(n) et so(p,q)}
If $\mathfrak{g}=\mathfrak{sl}(n,\mathbb{R})$ with $n \neq 4$ or $\mathfrak{so}(p,q)$ with $p+q=n$ and $(p,q)\neq (3,3)$, and $\alpha \in \Aut(\mathfrak{g})$, then
\[ \dim S^{\alpha} \leqslant \dim S -  \rk_{\mathbb{R}} G \]
with equality if and only if $n$ is odd and $S^{\alpha}=S^A$ with $A$ conjugated to
\[ Q_{n-1,1}=\begin{pmatrix}
-I_{n-1} & 0 \\
0 & 1 \\
\end{pmatrix} \]
in $\PSL(n,\mathbb{R})$ or $\PSO(p,q)$, or $n$ is even and $\alpha$ is conjugated to the outer automorphism corresponding to the conjugation by $Q_{n-1,1}$.
\end{lemma}

By abusing notations we will still denote $S^A$ the fixed point set by the automorphism corresponding to the conjugation by $A$ with $A$ conjugated to $Q_{n-1,1}$ even if it is not an inner automorphism.

\begin{proof}
We begin with the case $\mathfrak{g}\neq \mathfrak{so}(4,4)$ and we use the same strategy that for the proof of Lemma \ref{Méthode groupe complexe 2}. An automorphism $\alpha$ of $\mathfrak{g}$ is the composition of an inner automorphism and an outer automorphism $\rho$ of order 2 or 4. If $\rho$ has order 4 then $\alpha^2$ is the composition of an inner automorphism and $\rho^2$ which is of order 2, and we have the inclusion $S^{\alpha} \subset S^{\alpha^2}$ so we can replace $\alpha$ by $\alpha^2$ and it will suffice to consider the outer automorphisms of order 2. Then if $\rho$ has order 2, $\alpha^2$ is an inner automorphism, that is $\alpha^2=\Ad(A)$ with $A \in G_{ad}$. If $A$ is not trivial, by the computations in sections 7 and 8 in \cite{Aramayona} we have
\[ \dim S^{\alpha^2} \leqslant \dim S^A \leqslant \dim S - \rk_{\mathbb{R}} G \]
with equality in the last inequality if and only if $n$ is odd and $A$ conjugated by $Q_{n-1,1}$. The first inequality is an equality if and only if $S^{\alpha^2}=S^A$. We have proved the claim if $A$ is not trivial.
If $A$ is trivial, then $\alpha$ is an involution, so we use the classification of local symmetric spaces. For instance if $\mathfrak{g}=\mathfrak{sl}(n,\mathbb{R})$ with $n \neq 4$, the non-compact associated isotropy algebra $\mathfrak{h}=\mathfrak{g}^{\alpha}$ are $\mathfrak{so}(k,n-k)$, $\mathfrak{s}(\mathfrak{gl}(k,\mathbb{R}) \oplus \mathfrak{gl}(n-k,\mathbb{R}))$, $\mathfrak{gl}\left(\frac{n}{2},\mathbb{C}\right)$ or $\mathfrak{sp}(n,\mathbb{R})$, the last two cases only appearing if $n$ is even. In all theses cases we have
\[ \dim S^{\alpha} \leqslant \dim S - \rk_{\mathbb{R}} G \]
with equality if and only if $\mathfrak{h}=\mathfrak{s}(\mathfrak{gl}(1,\mathbb{R}) \oplus \mathfrak{gl}(n-1,\mathbb{R}))$, which corresponds to an automorphism $\alpha$ conjugated to the inner automorphism $\Ad(Q_{n-1,1})$ if $n$ is odd or to the outer automorphism of conjugation by $Q_{n-1,1}$ if $n$ if even.

It remains to consider the case $\mathfrak{g}=\mathfrak{so}(4,4)$. In this case the group of outer automorphism is isomorphic to the symmetric group $\mathcal{S}_4$ so we can have elements of order 2, 3 or 4. If $\rho$ is an outer automorphism of order 2 or 4, and $\alpha=\Ad(A) \circ \tau$, we apply the same method using the classification of local symmetric spaces and we see that
\[ \dim S^{\alpha} \leqslant \dim S -  \rk_{\mathbb{R}} G \]
with equality if and only if $\alpha$ is an outer automorphism corresponding to the conjugation by a matrix conjugated to $Q_{7,1}$.
If $\rho$ is of order 3 and $\alpha=\Ad(A) \circ \tau$ then $\alpha^3$ is inner and we have just to treat the case where it is trivial, that is $\alpha$ is of order 3. Then its complexification $\alpha^{\mathbb{C}}$ is an order 3 complex automorphism of $\mathfrak{g}^{\mathbb{C}}=\mathfrak{so}(8,\mathbb{C})$, a case already treated in the previous section. We know that the fixed point set $(\mathfrak{g}^{\mathbb{C}})^{\alpha^{\mathbb{C}}}$ is isomorphic to $\mathfrak{g}_2^{\mathbb{C}}$, $\mathfrak{sl}(3,\mathbb{C})$ or is compact. As $\mathfrak{g}^{\alpha}$ is a real form of $(\mathfrak{g}^{\mathbb{C}})^{\alpha^{\mathbb{C}}}$, it is isomorphic to $\mathfrak{g}_{2(2)}$, $\mathfrak{sl}(3,\mathbb{R})$, $\mathfrak{su}(2,1)$, $\mathfrak{s(u}(d_p,d_q) \oplus \mathfrak{o}(s_p,s_q))$ with $2d_p+s_p=4$ and $2d_q+s_q=4$ or is compact. In all these cases we have 
\[ \dim S^{\alpha} < \dim S- \rk_{\mathbb{R}} G \]
so we have proved the claim.
\end{proof}

Let us assume for a while that $\mathfrak{g} \neq \mathfrak{sl}(3,\mathbb{R})$ and $\mathfrak{g} \neq \mathfrak{sl}(4,\mathbb{R}) \cong \mathfrak{so}(3,3)$. We will conclude that $\vcd(\Gamma)=\underline{\gd}(\Gamma)$ using Lemma \ref{Lemme secours}. The first condition in the said lemma holds by Lemma \ref{Lemme sl(n) et so(p,q)}. To check the second condition, take $S^{\alpha}$ and $S^{\beta}$ maximal and distinct. We want to establish
\[ \dim(S^{\alpha} \cap S^{\beta}) \leqslant \vcd(\Gamma)-2. \]
First remark that by maximality $S^{\alpha}$ and $S^{\beta}$ are not contained in each other. If one of them is not of the form $S^A$ with $A$ conjugated to $Q_{n-1,1}$, let us say $S^{\alpha}$, then $\dim S^{\alpha} \leqslant \vcd(\Gamma)-1$ and $S^{\alpha} \cap S^{\beta}$ is a strict submanifold of $S^{\alpha}$ so the result holds. If we have $S^{\alpha}=S^A$ and $S^{\beta}=S^B$ with $A$ and $B$ conjugated to $Q_{n-1,1}$, then we refer to the computations in the proofs of Lemma 7.2 and Lemma 8.4 in \cite{Aramayona}. The proof of the third point is the same as for Lemma 7.5 and Lemma 8.8 in \cite{Aramayona}. Note that in \cite{Aramayona} the authors consider only inner automorphisms, so the case $n$ odd, but their argument also works without modifications of any kind for $n$ even.

It must be enlighted why the argument we just gave fails for $\mathfrak{sl}(3,\mathbb{R})$ and $\mathfrak{sl}(4,\mathbb{R}) \cong \mathfrak{so}(3,3)$. For $\mathfrak{sl}(3,\mathbb{R})$, the second condition of Lemma \ref{Lemme secours} does not hold anymore. In the case that $\mathfrak{sl}(4,\mathbb{R}) \cong \mathfrak{so}(3,3)$, the conclusion of Lemma \ref{Lemme sl(n) et so(p,q)} does not apply, because we have that
\[ \dim S^{\alpha}=\dim S-\rk_{\mathbb{R}} G \]
when $\alpha$ is either the conjugation by $Q_{5,1}$ in $\PSO(3,3)$ or the conjugation by $Q_{3,1}$ in $\PSL(4,\mathbb{R})$. These two are not conjugated, as the conjugation by $Q_{5,1}$ in $\PSO(3,3)$ corresponds in $\PSL(4,\mathbb{R})$ to an outer automorphism whose fixed point set is isomorphic to $\PSp(4,\mathbb{R})$. 

However the proof of Lemma 7.6 in \cite{Aramayona} concerning lattices in $\SL(3,\mathbb{R})$ can be adapted to $\Aut(\mathfrak{sl}(3,\mathbb{R}))$ and  $\Aut(\mathfrak{sl}(4,\mathbb{R}))$. In fact it can be adapted to $\Aut(\mathfrak{sl}(n,\mathbb{R}))$ for all $n \geqslant 3$, because a lattice in $\PSL(n,\mathbb{R})$ of $\mathbb{Q}$-rank $n-1$ can be conjugated to a lattice commensurable to $\PSL(n,\mathbb{Z})$ (see the classification of arithmetic groups of classical groups in Section 18.5 in \cite{Witte}).

As a result Proposition \ref{Proposition groupe réel} holds for all real classical simple Lie algebras.

\vspace{0.5cm}

\textbf{Lie algebra $\mathfrak{e}_{6(6)}$}

Here we consider the simple exceptional Lie algebra $\mathfrak{g}=\mathfrak{e}_{6(6)}$. Its outer automorphism group is of order 2 and its adjoint group is $G_{ad}=E_{6(6)}$, which is the group of real points of an algebraic group of real rank 6. This group contains a maximal compact subgroup $K$ isogenous to $\Sp(4)$. We will use Lemma \ref{Méthode groupe réel} to check the first condition of Lemma \ref{Méthode groupe complexe 2}.

The group $G_{ad}=E_{6(6)}$ contains a subgroup $\bar{G}$ isogenous to $\Sp(2,2)$ whose maximal compact subgroup is $\bar{K}=\Sp(2) \times \Sp(2)$. We see in \cite{Berger} that $E_{6(6)}/\Sp(2,2)$ is an irreducible symmetric space, furthermore we have $\rk(K)=\rk(\bar{K})=4$ and 
  \[ \dim \bar{S} = 16 < 42 - 6 = \dim S - \rk_{\mathbb{R}} G_{ad} \]
where $S$ and $\bar{S}$ are the Riemannian symmetric spaces associated to respectively $G_{ad}$ and $\bar{G}$.

Moreover, if $A \in \bar{K}$ is of finite order and non central,  we get from inequality \eqref{Equation 2 réelle}
\[ \dim \bar{S}^A \leqslant 8  < 16 - 6 = \dim \bar{S} -\rk_{\mathbb{R}} G_{ad}.  \]

Lemma \ref{Méthode groupe réel} applies and shows that the first condition of Lemma \ref{Méthode groupe complexe 2} holds. 

To check the second condition we  list the local symmetric spaces associated to an involution $\rho \in \Aut(\mathfrak{e}_{6(6)})$. By the classification of Berger in \cite{Berger}, the only non compact cases are when $\mathfrak{g}^{\rho}$ is isomorphic to $\mathfrak{sp}(2,2), \mathfrak{sp}(8,\mathbb{R}), \mathfrak{sl}(6,\mathbb{R}) \oplus \mathfrak{sl}(2,\mathbb{R}), \mathfrak{sl}(3,\mathbb{H}) \oplus \mathfrak{su}(2), \mathfrak{so}(5,5) \oplus \mathfrak{so}(1,1)$ or $\mathfrak{f}_{4(4)}$. We have in all cases
\[ \dim S^{\rho} < \dim S- \rk_{\mathbb{R}} G. \]

So by Lemma \ref{Méthode groupe complexe 2}, Proposition \ref{Proposition groupe réel} holds for $G=\Aut(\mathfrak{e}_{6(6)})$.

\vspace{0.5cm}

\textbf{Lie algebra $\mathfrak{e}_{6(2)}$}

Here we consider the simple exceptional Lie algebra $\mathfrak{g}=\mathfrak{e}_{6(2)}$. Its outer automorphism group is of order 2 and its adjoint group is $G_{ad}=E_{6(2)}$, which is the group of real points of an algebraic group of real rank 4. This group contains a maximal compact subgroup $K$ isogenous to $\SU(6)\times \SU(2)$. We will use Lemma \ref{Méthode groupe réel} to check the first condition of Lemma \ref{Méthode groupe complexe 2}.

The group $G_{ad}=E_{6(2)}$ contains a subgroup $\bar{G}$ isogenous to $\SO^*(10) \times \SO(2)$ whose maximal compact subgroup is $\bar{K}=\U(5) \times \SO(2)$. We see in \cite{Berger} that $E_{6(2)}/(\SO^*(10)\times \SU(2))$ is an irreducible symmetric space, furthermore we have $\rk(K)=\rk(\bar{K})=6$ and 
  \[ \dim \bar{S} = 20 < 40 - 4 = \dim S - \rk_{\mathbb{R}} G_{ad} \]
where $S$ and $\bar{S}$ are the Riemannian symmetric spaces associated to respectively $G_{ad}$ and $\bar{G}$.

Moreover, if $A \in \bar{K}$ is of finite order and non central,  we get from inequality \eqref{Equation 3 réelle}
\[ \dim \bar{S}^A \leqslant 12  < 20 - 4 = \dim \bar{S} -\rk_{\mathbb{R}} G_{ad}.  \]

Lemma \ref{Méthode groupe réel} applies and shows that the first condition of Lemma \ref{Méthode groupe complexe 2} holds. 

To check the second condition we  list the local symmetric spaces associated to an involution $\rho \in \Aut(\mathfrak{e}_{6(2)})$. By the classification of Berger in \cite{Berger}, the only non compact cases are when $\mathfrak{g}^{\rho}$ is isomorphic to $\mathfrak{sp}(1,3), \mathfrak{sp}(8,\mathbb{R}), \mathfrak{su}(2,4) \oplus \mathfrak{su}(2), \mathfrak{su}(3,3) \oplus \mathfrak{sl}(2,\mathbb{R}), \mathfrak{so}(4,6) \oplus \mathfrak{so}(2)$, $\mathfrak{so}^*(10) \oplus \mathfrak{so}(2)$ or $\mathfrak{f}_{4(4)}$. We have in all cases
\[ \dim S^{\rho} < \dim S- \rk_{\mathbb{R}} G. \]

So by Lemma \ref{Méthode groupe complexe 2}, Proposition \ref{Proposition groupe réel} holds for $G=\Aut(\mathfrak{e}_{6(2)})$.

\vspace{0.5cm}

\textbf{Lie algebra $\mathfrak{e}_{6(-14)}$}

Here we consider the simple exceptional Lie algebra $\mathfrak{g}=\mathfrak{e}_{6(-14)}$. Its outer automorphism group is of order 2 and its adjoint group is $G_{ad}=E_{6(-14)}$, which is the group of real points of an algebraic group of real rank 2. This group contains a maximal compact subgroup $K$ isogenous to $\SO(10)\times \SO(2)$. We will use Lemma \ref{Méthode groupe réel} to check the first condition of Lemma \ref{Méthode groupe complexe 2}.

The group $G_{ad}=E_{6(-14)}$ contains a subgroup $\bar{G}$ isogenous to $\SO^*(10) \times \SO(2)$ whose maximal compact subgroup is $\bar{K}=\U(5) \times \SO(2)$. We see in \cite{Berger} that $E_{6(-14)}/(\SO^*(10)\times \SU(2))$ is an irreducible symmetric space, furthermore we have $\rk(K)=\rk(\bar{K})=6$ and 
  \[ \dim \bar{S} = 20 < 32 - 2 = \dim S - \rk_{\mathbb{R}} G_{ad} \]
where $S$ and $\bar{S}$ are the Riemannian symmetric spaces associated to respectively $G_{ad}$ and $\bar{G}$.

Moreover, if $A \in \bar{K}$ is of finite order and non central,  we get from inequality \eqref{Equation 3 réelle}
\[ \dim \bar{S}^A \leqslant 12  < 20 - 2 = \dim \bar{S} -\rk_{\mathbb{R}} G_{ad}.  \]

Lemma \ref{Méthode groupe réel} applies and shows that the first condition of Lemma \ref{Méthode groupe complexe 2} holds. 

To check the second condition we  list the local symmetric spaces associated to an involution $\rho \in \Aut(\mathfrak{e}_{6(-14)})$. By the classification of Berger in \cite{Berger}, the only non compact cases are when $\mathfrak{g}^{\rho}$ is isomorphic to $\mathfrak{sp}(2,2), \mathfrak{su}(2,4) \oplus \mathfrak{su}(2), \mathfrak{su}(1,5) \oplus \mathfrak{sl}(2,\mathbb{R}), \mathfrak{so}(2,8) \oplus \mathfrak{so}(2)$, $\mathfrak{so}^*(10) \oplus \mathfrak{so}(2)$ or $\mathfrak{f}_{4(-20)}$. We have in all cases
\[ \dim S^{\rho} < \dim S- \rk_{\mathbb{R}} G. \]

So by Lemma \ref{Méthode groupe complexe 2}, Proposition \ref{Proposition groupe réel} holds for $G=\Aut(\mathfrak{e}_{6(-14)})$.

\vspace{0.5cm}

\textbf{Lie algebra $\mathfrak{e}_{6(-26)}$}

Here we consider the simple exceptional Lie algebra $\mathfrak{g}=\mathfrak{e}_{6(-26)}$. Its outer automorphism group is of order 2 and its adjoint group is $G_{ad}=E_{6(-26)}$, which is the group of real points of an algebraic group of real rank 2. This group contains a maximal compact subgroup $K$ isogenous to $F_4$. We will use Lemma \ref{Méthode groupe réel} to check the first condition of Lemma \ref{Méthode groupe complexe 2}.

The group $G_{ad}=E_{6(-26)}$ contains a subgroup $\bar{G}$ isogenous to $\Sp(1,3)$ whose maximal compact subgroup is $\bar{K}=\Sp(1)\times \Sp(3)$. We see in \cite{Berger} that $E_{6(-26)}/\Sp(1,3)$ is an irreducible symmetric space, furthermore we have $\rk(K)=\rk(\bar{K})=4$ and 
  \[ \dim \bar{S} = 12 < 26 - 2 = \dim S - \rk_{\mathbb{R}} G_{ad} \]
where $S$ and $\bar{S}$ are the Riemannian symmetric spaces associated to respectively $G_{ad}$ and $\bar{G}$.

Moreover, if $A \in \bar{K}$ is of finite order and non central,  we get from inequality \eqref{Equation 2 réelle}
\[ \dim \bar{S}^A \leqslant 8  < 12 - 2 = \dim \bar{S} -\rk_{\mathbb{R}} G_{ad}.  \]

Lemma \ref{Méthode groupe réel} applies and shows that the first condition of Lemma \ref{Méthode groupe complexe 2} holds. 

To check the second condition we  list the local symmetric spaces associated to an involution $\rho \in \Aut(\mathfrak{e}_{6(-26)})$. By the classification of Berger in \cite{Berger}, the only non compact cases are when $\mathfrak{g}^{\rho}$ is isomorphic to $\mathfrak{sp}(1,3), \mathfrak{sl}(3,\mathbb{H}) \oplus \mathfrak{sp}(1), \mathfrak{so}(1,9) \oplus \mathfrak{so}(1,1)$, or $\mathfrak{f}_{4(-20)}$. We have in all cases
\[ \dim S^{\rho} < \dim S- \rk_{\mathbb{R}} G. \]

So by Lemma \ref{Méthode groupe complexe 2}, Proposition \ref{Proposition groupe réel} holds for $G=\Aut(\mathfrak{e}_{6(-26)})$.

\vspace{0.5cm}

\textbf{Lie algebra $\mathfrak{e}_{7(7)}$}

Here we consider the simple exceptional Lie algebra $\mathfrak{g}=\mathfrak{e}_{7(7)}$. Its outer automorphism group is of order 2 and its adjoint group is $G_{ad}=E_{7(7)}$, which is the group of real points of an algebraic group of real rank 7. This group contains a maximal compact subgroup $K$ isogenous to $\SU(8)$. We will use Lemma \ref{Méthode groupe réel} to check the first condition of Lemma \ref{Méthode groupe complexe 2}.

The group $G_{ad}=E_{7(7)}$ contains a subgroup $\bar{G}$ isogenous to $E_{6(2)} \times \SO(2)$ whose maximal compact subgroup is $\bar{K}=\SU(6)\times \SU(2) \times \SO(2)$. We see in \cite{Berger} that $E_{7(7)}/(E_{6(2)} \times \SO(2))$ is an irreducible symmetric space, furthermore we have $\rk(K)=\rk(\bar{K})=7$ and 
  \[ \dim \bar{S} = 40 < 70 - 7 = \dim S - \rk_{\mathbb{R}} G_{ad} \]
where $S$ and $\bar{S}$ are the Riemannian symmetric spaces associated to respectively $G_{ad}$ and $\bar{G}$.

Moreover, if $A \in \bar{K}$ is of finite order and non central,  we get from the results about $\mathfrak{e}_{6(2)}$
\[ \dim \bar{S} - \dim \bar{S}^A \geqslant 20-12=8  > \rk_{\mathbb{R}} G_{ad}.  \]

Lemma \ref{Méthode groupe réel} applies and shows that the first condition of Lemma \ref{Méthode groupe complexe 2} holds. 

To check the second condition we  list the local symmetric spaces associated to an involution $\rho \in \Aut(\mathfrak{e}_{7(7)})$. By the classification of Berger in \cite{Berger}, the only non compact cases are when $\mathfrak{g}^{\rho}$ is isomorphic to $\mathfrak{su}(4,4), \mathfrak{sl}(8,\mathbb{R}), \mathfrak{sl}(4,\mathbb{H}),\mathfrak{so}(6,6) \oplus \mathfrak{sl}(2,\mathbb{R})$, $\mathfrak{so}^*(12)\oplus \mathfrak{sp}(1), \mathfrak{e}_{6(6)} \oplus \mathfrak{so}(1,1)$ or $\mathfrak{e}_{6(2)} \oplus \mathfrak{so}(2)$. We have in all cases
\[ \dim S^{\rho} < \dim S- \rk_{\mathbb{R}} G. \]

So by Lemma \ref{Méthode groupe complexe 2}, Proposition \ref{Proposition groupe réel} holds for $G=\Aut(\mathfrak{e}_{7(7)})$.

\vspace{0.5cm}

\textbf{Lie algebra $\mathfrak{e}_{7(-5)}$}

Here we consider the simple exceptional Lie algebra $\mathfrak{g}=\mathfrak{e}_{7(-5)}$. Its outer automorphism group is trivial so $G=\Aut(\mathfrak{g})$ is equal to the adjoint group $G_{ad}=E_{7(7)}$, which is the group of real points of an algebraic group of real rank 4. Thus we  only have to check the first condition of Lemma \ref{Méthode groupe complexe 2} and we will again use Lemma \ref{Méthode groupe réel}.

The group $G_{ad}=E_{7(-5)}$  contains a maximal compact subgroup $K$ isogenous to $\SO(12) \times \SU(2)$. It also contains a subgroup $\bar{G}$ isogenous to $\SU(4,4)$ whose maximal compact subgroup is $\bar{K}=\mathrm{S}(\U(4)\times \U(4))$. We see in \cite{Berger} that $E_{7(-5)}/\SU(4,4)$ is an irreducible symmetric space, furthermore we have $\rk(K)=\rk(\bar{K})=7$ and 
  \[ \dim \bar{S} = 32 < 64 - 4 = \dim S - \rk_{\mathbb{R}} G_{ad} \]
where $S$ and $\bar{S}$ are the Riemannian symmetric spaces associated to respectively $G_{ad}$ and $\bar{G}$.

Moreover, if $A \in \bar{K}$ is of finite order and non central, we get from inequality \eqref{Equation 1 réelle}
\[ \dim \bar{S}^A \leqslant 24  < 32 - 4 = \dim \bar{S} -\rk_{\mathbb{R}} G_{ad}.  \]

So by Lemma \ref{Méthode groupe réel} and Lemma \ref{Méthode groupe complexe 2}, Proposition \ref{Proposition groupe réel} holds for $G=\Aut(\mathfrak{e}_{7(-5)})$.

\vspace{0.5cm}

\textbf{Lie algebra $\mathfrak{e}_{7(-25)}$}

Here we consider the simple exceptional Lie algebra $\mathfrak{g}=\mathfrak{e}_{7(-25)}$. Its outer automorphism group is of order 2 and its adjoint group is $G_{ad}=E_{7(-25)}$, which is the group of real points of an algebraic group of real rank 3. This group contains a maximal compact subgroup $K$ isogenous to $E_6 \times \SO(2)$. We will use Lemma \ref{Méthode groupe réel} to check the first condition of Lemma \ref{Méthode groupe complexe 2}.

The group $G_{ad}=E_{7(-25)}$ contains a subgroup $\bar{G}$ isogenous to $\SU(2,6)$ whose maximal compact subgroup is $\bar{K}=\mathrm{S}(\U(2) \times \U(6))$. We see in \cite{Berger} that $E_{7(-25)}/\SU(2,6)$ is an irreducible symmetric space, furthermore we have $\rk(K)=\rk(\bar{K})=7$ and 
  \[ \dim \bar{S} = 24 < 54 - 3 = \dim S - \rk_{\mathbb{R}} G_{ad} \]
where $S$ and $\bar{S}$ are the Riemannian symmetric spaces associated to respectively $G_{ad}$ and $\bar{G}$.

Moreover, if $A \in \bar{K}$ is of finite order and non central, we get from inequality \eqref{Equation 1 réelle}
\[ \dim \bar{S}^A \leqslant 20  < 24 - 3 = \dim \bar{S} -\rk_{\mathbb{R}} G_{ad}.  \]

Lemma \ref{Méthode groupe réel} applies and shows that the first condition of Lemma \ref{Méthode groupe complexe 2} holds. 

To check the second condition we  list the local symmetric spaces associated to an involution $\rho \in \Aut(\mathfrak{e}_{7(-25)})$. By the classification of Berger in \cite{Berger}, the only non compact cases are when $\mathfrak{g}^{\rho}$ is isomorphic to $\mathfrak{su}(2,6), \mathfrak{sl}(4,\mathbb{H}),\mathfrak{so}(2,10) \oplus \mathfrak{sl}(2,\mathbb{R})$, $\mathfrak{so}^*(12)\oplus \mathfrak{sp}(1), \mathfrak{e}_{6(-14)} \oplus \mathfrak{so}(2)$ or $\mathfrak{e}_{6(-26)} \oplus \mathfrak{so}(1,1)$. We have in all cases
\[ \dim S^{\rho} < \dim S- \rk_{\mathbb{R}} G. \]

So by Lemma \ref{Méthode groupe complexe 2}, Proposition \ref{Proposition groupe réel} holds for $G=\Aut(\mathfrak{e}_{7(-25)})$.

\vspace{0.5cm}

\textbf{Lie algebra $\mathfrak{e}_{8(8)}$}

Here we consider the simple exceptional Lie algebra $\mathfrak{g}=\mathfrak{e}_{8(8)}$. Its outer automorphism group is trivial so $G=\Aut(\mathfrak{g})$ is equal to the adjoint group $G_{ad}=E_{8(8)}$, which is the group of real points of an algebraic group of real rank 8. Thus we only have to check the first condition of Lemma \ref{Méthode groupe complexe 2} and we will again use Lemma \ref{Méthode groupe réel}.

The group $G_{ad}=E_{8(8)}$  contains a maximal compact subgroup $K$ isogenous to $\SO(16)$. It also contains a subgroup $\bar{G}$ isogenous to $\SO^*(16)$ whose maximal compact subgroup is $\bar{K}=\U(8)$. We see in \cite{Berger} that $E_{8(8)}/\SO(16)$ is an irreducible symmetric space, furthermore we have $\rk(K)=\rk(\bar{K})=8$ and 
  \[ \dim \bar{S} = 56 < 128 - 8 = \dim S - \rk_{\mathbb{R}} G_{ad} \]
where $S$ and $\bar{S}$ are the Riemannian symmetric spaces associated to respectively $G_{ad}$ and $\bar{G}$.

Moreover, if $A \in \bar{K}$ is of finite order and non central, we get from inequality \eqref{Equation 3 réelle}
\[ \dim \bar{S}^A \leqslant 42  < 56 - 8 = \dim \bar{S} -\rk_{\mathbb{R}} G_{ad}.  \]

So by Lemma \ref{Méthode groupe réel} and Lemma \ref{Méthode groupe complexe 2}, Proposition \ref{Proposition groupe réel} holds for $G=\Aut(\mathfrak{e}_{8(8)})$.

\vspace{0.5cm}

\textbf{Lie algebra $\mathfrak{e}_{8(-24)}$}

Here we consider the simple exceptional Lie algebra $\mathfrak{g}=\mathfrak{e}_{8(-24)}$. Its outer automorphism group is trivial so $G=\Aut(\mathfrak{g})$ is equal to the adjoint group $G_{ad}=E_{8(-24)}$, which is the group of real points of an algebraic group of real rank 4. Thus we only have to check the first condition of Lemma \ref{Méthode groupe complexe 2} and we will again use Lemma \ref{Méthode groupe réel}.

The group $G_{ad}=E_{8(-24)}$  contains a maximal compact subgroup $K$ isogenous to $E_7 \times \SU(2)$. It also contains a subgroup $\bar{G}$ isogenous to $\SO^*(16)$ whose maximal compact subgroup is $\bar{K}=\U(8)$. We see in \cite{Berger} that $E_{8(-24)}/\SO(16)$ is an irreducible symmetric space, furthermore we have $\rk(K)=\rk(\bar{K})=8$ and 
  \[ \dim \bar{S} = 56 < 112 - 4 = \dim S - \rk_{\mathbb{R}} G_{ad} \]
where $S$ and $\bar{S}$ are the Riemannian symmetric spaces associated to respectively $G_{ad}$ and $\bar{G}$.

Moreover, if $A \in \bar{K}$ is of finite order and non central, we get from inequality \eqref{Equation 3 réelle}
\[ \dim \bar{S}^A \leqslant 42  < 56 - 4 = \dim \bar{S} -\rk_{\mathbb{R}} G_{ad}.  \]

So by Lemma \ref{Méthode groupe réel} and Lemma \ref{Méthode groupe complexe 2}, Proposition \ref{Proposition groupe réel} holds for $G=\Aut(\mathfrak{e}_{8(-24)})$.

\vspace{0.5cm}

\textbf{Lie algebra $\mathfrak{g}_{2(2)}$}

Here we consider the simple exceptional Lie algebra $\mathfrak{g}=\mathfrak{g}_{2(2)}$. Its outer automorphism group is trivial, so the group $G=\Aut(\mathfrak{g})$ equals the adjoint group $G_{ad}=G_{2(2)}$, which is the group of real points of an algebraic group of real rank 2. Thus we only have to check the conditions of Lemma \ref{Méthode groupe complexe}. 

The group $G_{2(2)}$ contains a maximal compact subgroup $K$ isomorphic to $\SO(4)$. It also contains a subgroup $\bar{G}$ isogenous to $\SL(2,\mathbb{R}) \times \SL(2,\mathbb{R})$ whose maximal compact subgroup is $\bar{K}=\SO(2) \times \SO(2)$. We see in \cite{Berger} that $G_{2(2)}/(\SL(2,\mathbb{R}) \times \SL(2,\mathbb{R}))$ is an irreducible symmetric space, furthermore we have $\rk(K)=\rk(\bar{K})=2$ and 
  \[ \dim \bar{S} = 4 < 8 - 2 = \dim S - \rk_{\mathbb{R}} G \]
where $S$ and $\bar{S}$ are the Riemannian symmetric spaces associated to respectively $G=G_{ad}$ and $\bar{G}$.

Moreover, if $A \in \bar{K}$ is of finite order and non central,  we have
\[ \dim \bar{S}^A \leqslant 2  = 4 - 2 = \dim \bar{S} -\rk_{\mathbb{R}} G. \]

The equality case in the last inequality happens when $A$ is conjugated to a matrix of the form:
\[ \begin{pmatrix}
\pm I_2 & 0 \\
0 & R_{\theta} \\
\end{pmatrix} \]
with
\[ R_{\theta}=\begin{pmatrix}
\cos(\theta) & -\sin(\theta) \\
\sin(\theta) & \cos(\theta) \end{pmatrix}. \]

Assume that $A$ is of this form and that the first block is $I_2$, we will prove directly that
\[ \dim S^A < \dim S - \rk_{\mathbb{R}} G. \] 

First of all, $C_K(A)=C_{\SO(4)}(A)= \SO(2) \times \SO(2)$ is of dimension 2.

To study $C_G(A)$, we have to know to which element of $G=G_{2(2)}$ this matrix corresponds to. Recall that $G_{2(2)}$ is the group of automorphisms of the (non associative) algebra $\mathbb{O}'$ of split octonions, which is of dimension 8 over $\mathbb{R}$, and equiped with a quadratic form of signature (4,4) (see section 1.13 of in \cite{Yokota}). We can decompose $\mathbb{O}'$ into the direct sum $\mathbb{H} \oplus \mathbb{H}e_4'$ where $\mathbb{H}=\vect{1,e_1,e_2,e_3}$ is the quaternion algebra.

So $G$ is a subgroup of the special orthogonal group $\SO(3,4)$ which preserves the standard form of signature (4,4) over $\mathbb{R}^8$ and fixes 1.

The maximal compact subgroup $K$ corresponds to the stabilizer of $\mathbb{H}$, meaning the elements $\alpha \in G$ such that $\alpha(\mathbb{H})=\mathbb{H}$. Automatically we have $\alpha(\mathbb{H}e_4')=\mathbb{H}e_4'$ as this is the orthogonal of $\mathbb{H}$.  $K$ is isomorphic to $\SO(4)$ via the isomorphism who sends $\alpha$ to its restriction to $\mathbb{H}e_4'$. 

Consequently, the matrix $A$ we consider corresponds to the matrix of the restriction of an element $\alpha \in K$ to $\mathbb{H}e_4'$. This element $\alpha$ is entirely determined by the matrix $A$, indeed for example we have:
\[ \alpha(e_1e_4')=\alpha(e_1)\alpha(e_4')=\alpha(e_1)e_4'=e_1e_4' \]
so we deduce $\alpha(e_1)=e_1$. Similarly we find:
\[ \alpha(e_2)=\cos(\theta)e_2+\sin(\theta)e_3, \]
\[ \alpha(e_3)=-\sin(\theta)e_2 + \cos(\theta) e_3. \]
Knowing $\alpha(1)=1$, we have completely described $\alpha$. The matrix of $\SO(3,4)$ which corresponds to is:
\[ \tilde{A}=\begin{pmatrix}
1 & 0 & 0 & 0 \\
0 & R_{\theta} & 0 & 0 \\
0 & 0 & I_2 & 0 \\
0 & 0 & 0 & R_{\theta} \\
\end{pmatrix}. \]
Then we remark that $C_G(A) \subset C_{\SO(3,4)}(\tilde{A})$ so:
\[ \dim C_G(A) \leqslant \dim C_{\SO(3,4)}(\tilde{A})= \dim \mathrm{S}(\mathrm{O}(1,2) \times \mathrm{U}(1,1))=6. \]
Finally:
\[ \dim S^A \leqslant 6-2=4 < 6= \dim S - \rk_{\mathbb{R}} G. \]

Thus we have that 
\[ \dim S^A < \dim S - \rk_{\mathbb{R}} G \]
for every $A \in G$ of finite order and non central.
Then by Lemma \ref{Lemme clé}, Proposition \ref{Proposition groupe réel} holds for $\mathfrak{g}=\mathfrak{g}_{2(2)}$.

\vspace{0.5cm}

\textbf{Lie algebra $\mathfrak{f}_{4(4)}$}

Here we consider the simple exceptional Lie algebra $\mathfrak{g}=\mathfrak{f}_{4(4)}$. Its outer automorphism group is trivial, so the group $G=\Aut(\mathfrak{g})$ equals the adjoint group $G_{ad}=F_{4(4)}$, which is the group of real points of an algebraic group of real rank 4. Thus we only have to check the conditions of Lemma \ref{Méthode groupe complexe}. 

The group $F_{4(4)}$ contains a maximal compact subgroup $K$ isogenous to $\Sp(3) \times \Sp(1)$. It also contains a subgroup $\bar{G}$ isogenous to $\SO(4,5)$ whose maximal compact subgroup is $\bar{K}=\mathrm{S}(\mathrm{O}(4) \times \mathrm{O}(5))$. We see in \cite{Berger} that $F_{4(4)}/\SO(4,5)$ is an irreducible symmetric space, furthermore we have $\rk(K)=\rk(\bar{K})=4$ and 
  \[ \dim \bar{S} = 20 < 28 - 4 = \dim S - \rk_{\mathbb{R}} G\]
where $S$ and $\bar{S}$ are the Riemannian symmetric spaces associated to respectively $G=G_{ad}$ and $\bar{G}$.

Moreover, if $A \in \bar{K}$ is of finite order and non central,  we have by the computations in section 8 of \cite{Aramayona}
\[ \dim \bar{S}^A \leqslant 16  = 20 - 4 = \dim \bar{S} -\rk_{\mathbb{R}} G.  \]

The equality case in the last inequality happens when $A$ is conjugated to the matrix:
\[ \begin{pmatrix}
-I_8 & 0 \\
0 & 1 \\
\end{pmatrix}  \in \SO(4)\times \SO(5). \]

Assuming that $A$ is of this form, the conjugation by $A$ is an involutive automorphism of $G=G_{ad}=F_{4(4)}$ so the quotient $G/G^A$ is a symmetric space and we know by the classification in \cite{Berger} that $G^A$ is isogenous to either $\Sp(6,\mathbb{R})\times \Sp(2,\mathbb{R})$, $\Sp(2,1)\times \Sp(1)$ or $\SO(4,5)$. In all these cases the inequality
\[ \dim S^A < \dim S - \rk_{\mathbb{R}} G  \]
holds (in fact $G^A$ is isogenous to $\SO(4,5)$).

Thus we have that 
\[ \dim S^A < \dim S - \rk_{\mathbb{R}} G \]
for every $A \in G$ of finite order and non central.
Then by Lemma \ref{Lemme clé}, Proposition \ref{Proposition groupe réel} holds for $\mathfrak{g}=\mathfrak{f}_{4(4)}$, and it concludes the proof.
\end{proof}

\section{Semisimple Lie algebras}

We prove in this section the Main Theorem:

\begin{mth}
Let $\mathfrak{g}$ be a semisimple Lie algebra and $G=\Aut(\mathfrak{g})$. Then 
\[ \underline{\gd}(\Gamma)=\vcd(\Gamma) \]
for every lattice $\Gamma \subset G$.
\end{mth}

Recall that if $\mathfrak{g}$ is semisimple, it is isomorphic to a sum of simple Lie algebras $\mathfrak{g}_1  \oplus \dots \oplus \mathfrak{g}_r$. The adjoint group $G_{ad}$ of $\mathfrak{g}$ is then isomorphic to a product of simple Lie groups, that is $G_{ad}=G_1 \times G_2  \times \dots \times G_n$ where the $G_i$ are the adjoint groups of the $\mathfrak{g}_i$. We can also assume that $G_{ad}$ has no compact factors, indeed the symmetric spaces $S$ and $S^{\alpha}$ do not change if we replace $G_{ad}$ by its quotient by the compact factors. An automorphism of $\mathfrak{g}$ is the composition of a permutation $\sigma$ of the isomorphic factors of $\mathfrak{g}$ and a diagonal automorphism $\rho$ of the form $\rho=\rho_1 \oplus \dots \oplus \rho_r$ with $\rho_i \in \Aut(\mathfrak{g}_i)$. 

We explain now why does the strategy used in the previous sections not work. The point is that the inequality
\[ \dim S^{\alpha} \leqslant \dim S -\rk_{\mathbb{R}} G \]
for $\alpha \in \Aut(\mathfrak{g})$ needed to apply Lemma \ref{Lemme secours} does not hold even in the simplest cases.

In fact, if $G_{ad}=\SL(3,\mathbb{R}) \times \SL(3,\mathbb{R})$ and $A=\left(I_3,\begin{pmatrix}
-1 & 0 &0 \\
0 & -1 & 0\\
0 & 0 & 1 \\
\end{pmatrix} \right)$, we have $\dim S^A=5+3=8 > 6= \dim S-\rg_{\mathbb{R}}G$.

We bypass this this problem by improving the lower bound
\[ \vcd(\Gamma) \geqslant \dim S-\rk_{\mathbb{R}} G \]
used above.
Remember that by Theorem \ref{Theoreme Borel-Serre}
\[ \vcd(\Gamma)=\dim S-\rk_{\mathbb{Q}}\Gamma \]
as long as $\Gamma$ is arithmetic, so we want to majorate $\rk_{\mathbb{Q}}\Gamma$.
To do that we will restrict our study to irreducible lattices. Recall that in this context, a lattice $\Gamma$ in $G$ is irreducible if $\Gamma H$ is dense in $G$ for every non-compact, closed, normal subgroup $H$ of $G_{ad}$.

We prove the following result, which is probably known to experts:

\begin{prop}\label{Prop rang rationnel}
Let $\mathfrak{g}=\mathfrak{g}_1\oplus \dots \oplus \mathfrak{g}_n$ be a semisimple Lie algebra and $G_i$ the adjoint group of $\mathfrak{g}_i$ for $i=1\dots n$. Then
\[ \rk_{\mathbb{Q}}\Gamma \leqslant \min_{i=1\dots n} \rk_{\mathbb{R}} G_i \]
for every irreducible arithmetic lattice $\Gamma \subset G=\Aut(\mathfrak{g})$.
\end{prop}

Proposition \ref{Prop rang rationnel} will follow from the following theorem proved in \cite{Johnson}

\begin{thm}\label{Theoreme Q-simple}
Let $G=G_1 \times \dots \times G_n$ ($n \geqslant 2$) be a product of non-compact connected simple Lie groups. The following statements are equivalent:
\begin{enumerate}
\item $G$ contains an irreducible lattice.
\item $G$ is isomorphic as a Lie group to $(\mathbb{G}_{\mathbb{R}})^0$, where $\mathbb{G}$ is an $\mathbb{Q}$-simple algebraic group.
\item $G$ is isotypic, that is the complexifications of the Lie algebras of the $G_i$ are isomorphic.
\end{enumerate}
In addition to that, in this case $G$ contains both cocompact and non cocompact irreducible lattices.
\end{thm}

Recall that an algebraic group $\mathbb{G}$ defined over $\mathbb{Q}$ is said to be \textit{$\mathbb{Q}$-simple} if it does not contain non-trivial connected normal subgroups defined over $\mathbb{Q}$.

Then we can prove Proposition \ref{Prop rang rationnel}.

\begin{proof}[Proof of Proposition \ref{Prop rang rationnel}]

Remark that $\rk_{\mathbb{Q}}\Gamma=\rk_{\mathbb{Q}}(\Gamma \cap G_{ad})$. Moreover if $\Gamma$ is an irreducible lattice of $G$ then $\Gamma \cap G_{ad}$ is an irreducible lattice of $G_{ad}$, so we can assume that $\Gamma \subset G_{ad}$.
Remember that $G_{ad}=G_1\times \dots \times G_n$ and that we assumed that none of the $G_i$ is compact. If $n=1$ the result is trivial so assume that $n\geqslant 2$. Then $\rk_{\mathbb{R}} G_{ad} \geqslant 2$ so $\Gamma$ is arithmetic by Theorem \ref{Théorème Margulis} and there exists a Lie group isomorphism $\varphi: G_{ad} \rightarrow (\mathbb{G}_{\mathbb{R}})^0$, where $\mathbb{G}$ is $\mathbb{Q}$-simple by Theorem \ref{Theoreme Q-simple}. Then we have $\rk_{\mathbb{Q}} \Gamma=\rk_{\mathbb{Q}} \mathbb{G}$. The algebraic group $\mathbb{G}$ is isomorphic to a product $\mathbb{G}_1 \times \dots \times \mathbb{G}_n$ where the $\mathbb{G}_i$ are $\mathbb{R}$-groups with $(\mathbb{G}_i)_{\mathbb{R}}$ isomorphic to $G_i$ for $i=1\dots n$ (we can define $\mathbb{G}_i$ as the centralizer in $\mathbb{G}$ of the product $\prod_{k \neq i} G_k$). We note $\pi_i$ the canonical projection of $\mathbb{G}$ on $\mathbb{G}_i$. Let $\mathbb{T} \subset \mathbb{G}$ be maximal $\mathbb{Q}$-split torus.  Our goal is to prove that the restriction $\pi_i |_{\mathbb{T}}$ is of finite kernel. On the one hand, $\ker(\pi_i|_{\mathbb{G}_{\mathbb{Q}}}) \subset \mathbb{G}_{\mathbb{Q}}$ is a normal subgroup of $\mathbb{G}_{\mathbb{Q}}$ and $\mathbb{G}$ is defined over $\mathbb{Q}$, so the Zariski closure of $\ker(\pi_i|_{\mathbb{G}_{\mathbb{Q}}})$ is defined over $\mathbb{Q}$ by the Galois rationality criterion. However it is a non trivial normal subgroup of $\mathbb{G}$ which is $\mathbb{Q}$-simple so $\ker(\pi_i|_{\mathbb{G}_{\mathbb{Q}}})$ is finite (it may be not connected). So  $\ker(\pi_i|_{\mathbb{T}_{\mathbb{Q}}})$ is finite too. But  $\ker(\pi_i|_{\mathbb{T}})$ is a subgroup of the $\mathbb{Q}$-split torus $\mathbb{T}$, so its identity component is a $\mathbb{Q}$-split torus, and we have just seen its group of rational points is finite, so $\ker(\pi_i|_{\mathbb{T}})$  is finite too.

Then the image of $\mathbb{T}$ by $\pi_i$ is a torus of $\mathbb{G}_i$ of the same dimension than $\mathbb{T}$ (see \cite[Cor. 8.4 p.114]{Borel}). It may not be $\mathbb{Q}$-split because the projection is not defined over $\mathbb{Q}$, but it is $\mathbb{R}$-split as the projection is defined over $\mathbb{R}$, so:
\[ \rg_{\mathbb{Q}} \mathbb{G}=\dim \mathbb{T} \leqslant \rg_{\mathbb{R}} \mathbb{G}_i=\rg_{\mathbb{R}} G_i. \]
\end{proof}

We can now conclude the proof of our Main Theorem. 

\begin{proof}[Proof of the Main Theorem]
If $\mathfrak{g}$ is simple then the result follows from Propositions \ref{Prop rang 1}, \ref{Prop groupe complexe} and \ref{Proposition groupe réel}. Then we assume that $\mathfrak{g}=\mathfrak{g}_1\oplus \dots \oplus \mathfrak{g}_n$ with $n \geqslant 2$. We can also assume that the adjoint group $G_{ad}$ of $\mathfrak{g}$ is of the form $G_{ad}=G_1\times \dots \times G_n$ where the $G_i$ are simple, non-compact and $G_i$ is the adjoint group of $\mathfrak{g}_i$. 

We begin with the case where $\Gamma$ is irreducible. Then we have
\[ \rk_{\mathbb{Q}}\Gamma \leqslant \min_{i=1\dots r} \rk_{\mathbb{R}} G_i. \]
As $\rk_{\mathbb{R}} G \geqslant 2$, $\Gamma$ is arithmetic by Theorem \ref{Théorème Margulis}. Remember that $\Gamma \cap G_{ad}$ is also an irreducible arithmetic lattice of $G_{ad}$. We can then assume that $G_{ad}=(\mathbb{G}_{\mathbb{R}})^0$, where $\mathbb{G}=\mathbb{G}_1\times \dots \times \mathbb{G}_n$ is a semisimple $\mathbb{Q}$-group, which is $\mathbb{Q}$-simple by Theorem \ref{Theoreme Q-simple}. As $G_{ad}$ has trivial center, we can assume that $\mathbb{G}$ is centerfree. In this case we have $\Gamma \cap G_{ad} \subset \mathbb{G}_{\mathbb{Q}}$.

We want to use Lemma \ref{Lemme clé}. Let $\alpha \in \Gamma$ of finite order non central. Then $\alpha$ is of the form $\sigma \circ \rho$ where $\sigma$ is a permutation of the isomorphic factors of $\mathfrak{g}$ and $\rho=\rho_1 \oplus \dots \rho_r$ with $\rho_i \in \Aut(\mathfrak{g}_i)$. Assume for a while that $\sigma$ is trivial. We identify $\alpha \in \Aut(\mathfrak{g})$ (resp. $\rho_i \in \Aut(\mathfrak{g}_i)$) with the corresponding automorphism of $G_{ad}$ (resp. $G_i$). The key point is to remark that for all $i$ between 1 and $n$, the automorphism $\rho_i$ is not trivial. In fact if $A \in \Gamma \cap G_{ad}\subset \mathbb{G}_{\mathbb{Q}}$, we can identify it with the inner automorphism $\Ad(A)$ and we have
\[ \alpha \circ \Ad(A) \circ \alpha^{-1}=\Ad(\alpha(A))  \in \Gamma \cap G_{ad} \subset \mathbb{G}_{\mathbb{Q}}, \]
so $\alpha(A)$ lies also in $\mathbb{G}_{\mathbb{Q}}$. Recall that we have seen in the proof of Proposition \ref{Prop rang rationnel} that the projections $\pi_i|_{\mathbb{G}_{\mathbb{Q}}}: \mathbb{G}_{\mathbb{Q}} \rightarrow G_i$ are injective. So if $\rho_i$ is trivial, we have $\pi_i(A)=\pi_i(\alpha(A))$ for each $A \in \Gamma \cap G_{ad} \subset \mathbb{G}_{\mathbb{Q}}$, which leads to $\alpha(A)=A$. Then $\alpha$ is trivial on $\Gamma \cap G_{ad}$ which is Zariski-dense in $G_{ad}$, so $\alpha$ is trivial.

Finally, $\alpha=\rho=\rho_1 \oplus \dots \oplus \rho_n$ where each $\rho_i$ is a non trivial automorphism of $G_i$. By Proposition \ref{Prop rang rationnel} we also have
 \[ \rk_{\mathbb{Q}}\Gamma \leqslant \min_{i=1\dots n} \rk_{\mathbb{R}} G_i. \]
 
 Then if we note $S$ the symmetric space associated to $G$ and $S_i$ those associated to $G_i$ and we have by Propositions \ref{Prop groupe complexe} and \ref{Proposition groupe réel}
\begin{align*}
\dim S^{\alpha}= \sum_{i=1}^n \dim S_i^{\rho_i} &\leqslant \sum_{i=1}^n(\dim S_i -\rk_{\mathbb{R}} G_i) \\
 &\leqslant \dim S-\sum_{i=1}^n \rk_{\mathbb{R}} G_i \\
 &< \dim S-\rk_{\mathbb{Q}} \Gamma.
\end{align*} 
 as we assumed $n\geqslant 2$.
 
 By Theorem \ref{Theoreme Borel-Serre}, $\dim S^A < \vcd(\Gamma)$ and Lemma \ref{Lemme clé} gives us the result.

If $\sigma$ is not trivial, the fixed point set will be even smaller. Indeed, assume for simplicity that $\mathfrak{g}=\mathfrak{g}_1 \oplus \mathfrak{g}_2$ where $\mathfrak{g}_1$ and $\mathfrak{g}_2$ are isomorphic, and $\alpha \in \Aut(\mathfrak{g})$ is of the form
\[ \alpha(x_1,x_2)=(\rho_1(x_2),\rho_2(x_1)) \]
with $\rho_1,\rho_2 \in \Aut(\mathfrak{g}_1)=\Aut(\mathfrak{g}_2)$. Then the fixed point set $S^{\alpha}$ is $S_1^{\rho_1\rho_2}$ where $S_1$ is the symmetric space associated to $\mathfrak{g}_1$. In fact the elements fixed by $\alpha$ are of the form $(x_1,\rho_2(x_1))$ where $x_1$ is a fixed point of $\rho_1\rho_2$. So we have
\[ \dim S^{\alpha} \leqslant \dim S_1 < \dim S - \rk_{\mathbb{R}} G_1 \leqslant \dim S - \rk_{\mathbb{Q}}  \Gamma=\vcd(\Gamma) \]
as $\dim S=2 \dim S_1$ and $\dim S_1 > \rk_{\mathbb{R}} G_1$. The same argument works for a higher number of summands by decomposing the permutation $\sigma$ into disjoint cycles.

Finally if $\Gamma$ is reducible, there exists a decomposition of $G=H_1 \times H_2$ such that the projections $\pi_1(\Gamma)$ and $\pi_2(\Gamma)$ are lattices in $H_1$ and $H_2$, and then $\Gamma \subset \pi_1(\Gamma) \times \pi_2(\Gamma)$ (see the proof of \cite[Prop. 4.3.3]{Witte}). It follows by induction that $\Gamma$ is contained in a product of irreducible lattices of factors of $G$. We will treat the case where $G=H_1 \times H_2$, $\Gamma \subset \Gamma_1 \times \Gamma_2$ and $\Gamma_1$ and $\Gamma_2$ are irreducible lattices of $H_1$ and $H_2$. As $\Gamma$ is of finite index in $G$, it is also of finite index in $\Gamma_1 \times \Gamma_2$ so $\vcd(\Gamma)=\vcd(\Gamma_1 \times \Gamma_2)$. If we note $S$, $S_1$, $S_2$ the symmetric spaces associated to $G$, $H_1$, $H_2$, by Theorem \ref{Theoreme Borel-Serre} we have:
\begin{align*}
\vcd(\Gamma)=\vcd(\Gamma_1 \times \Gamma_2)&=\dim S_1 + \dim S_2-\rk_{\mathbb{Q}} \Gamma _1 -\rk_{\mathbb{Q}} \Gamma_2 \\ &=\vcd(\Gamma_1)+\vcd(\Gamma_2).
\end{align*}
Finally, we have:
 \[ \underline{\gd}(\Gamma) \leqslant \underline{\gd}(\Gamma_1 \times \Gamma_2) \leqslant \underline{\gd}(\Gamma_1)+\underline{\gd}(\Gamma_2)\]
 because if $X_1$ and $X_2$ are models for $\underline{E}\Gamma_1$ and $\underline{E}\Gamma_2$, $X_1 \times X_2$ is a model for $\underline{E}(\Gamma_1\times \Gamma_2)$. As $\Gamma_1$ and $\Gamma_2$ are irreducible, we have $\underline{\gd}(\Gamma_1)=\vcd(\Gamma_1)$ and $\underline{\gd}(\Gamma_2)=\vcd(\Gamma_2)$ so:
\[ \underline{\gd}(\Gamma) \leqslant \vcd(\Gamma). \]

The other inequality is always true, so it concludes the proof of the Main Theorem.
\end{proof}

We will end with the proof of Corollaries \ref{Corollaire modèle cocompact} and \ref{Corollaire commensurable}.

\begin{proof}[Proof of Corollary \ref{Corollaire modèle cocompact}]
The case of real rank 1 is treated in Proposition 2.6 in \cite{Aramayona}. For higher real rank, we know by the Main Theorem that there exists a model for $\underline{E}\Gamma$ of dimension $\vcd(\Gamma)$. We also know that the Borel-Serre bordification is a cocompact model for $\underline{E}\Gamma$. Then using the same construction as in the proof of Corollary 1.1 in \cite{Aramayona}, one has a cocompact model for $\underline{E}\Gamma$ of dimension $\vcd(\Gamma)$. As all models of $\underline{E}\Gamma$ are $\Gamma$-equivariantly homotopy equivalent and the symmetric space $S$ is also a model for $\underline{E}\Gamma$, we conclude that $S$ is $\Gamma$-equivariantly homotopy equivalent to a cocompact model for $\underline{E}\Gamma$ of dimension $\vcd(\Gamma)$.
\end{proof}

\begin{proof}[Proof of Corollary \ref{Corollaire commensurable}]

We have to prove that if $\Gamma_1 \subset \Aut(\mathfrak{g})$ and $\Gamma_2$ have a common subgroup $\tilde{\Gamma}$ of finite index, then  $\underline{\gd}(\Gamma_2)=\vcd(\Gamma_2)$. To that end, we will prove that $\Gamma_2$ is essentially also a lattice in $\Aut(\mathfrak{g})$. First note that $\tilde{\Gamma}$ is a lattice in $\Aut(\mathfrak{g})$, so we can assume that $\tilde{\Gamma}=\Gamma_1 \subset \Gamma_2$. We can also assume that $\Gamma_1$ is a normal finite index subgroup of $\Gamma_2$. Then $\Gamma_2$ acts by conjugation on $\Gamma_1$. By Mostow rigidity Theorem (see for example \cite[Thm. 15.1.2]{Witte}), automorphisms of $\Gamma_1$ can be extended to automorphisms of $G_{ad}$, so we have a morphism $\Gamma_2 \rightarrow \Aut(G_{ad})$. The kernel $N$ of this morphism does not intersect $\Gamma_1$ (since $\Gamma_1$ is centerfree, as it is a lattice and thus it is Zariski-dense in $G_{ad}$) and $\Gamma_1$ is of finite index in $\Gamma_2$, so $N$ is finite. Then $\Gamma_2/N$ is isomorphic to a lattice in $\Aut(G_{ad})$. The result follows now from the Main Theorem and Lemma \ref{Finite normal subgroups}.

Note that Mostow rigidity theorem does not apply to the group $\PSL(2,\mathbb{R})$, whose associated symmetric space is the hyperbolic plane. In this case the lattice $\Gamma_1$ is either a virtually free group or a virtually surface group. In the first case the group $\Gamma_2$ is also virtually free, so there exists a model for $\underline{E}\Gamma_2$ which is a tree (see \cite{Karass}), and $\underline{\gd}(\Gamma_2)=\vcd(\Gamma_2)=1$. In the second case, $\Gamma_2$ acts as a convergence group on $\mathbb{S}^1=\partial_{\infty} \Gamma_1$, so it is also a Fuchsian group (see \cite{Gabai}), that is $\Gamma_2$ is isomorphic to a cocompact lattice of $\PSL(2,\mathbb{R})$. Finally we have $\underline{\gd}(\Gamma_2)=\vcd(\Gamma_2)=2$.
\end{proof}

\bibliographystyle{plain}
\bibliography{biblio_semisimple}
\nocite{Yokota}
\nocite{Aramayona}
\nocite{Helgason}
\nocite{Berger}
\nocite{Johnson}
\nocite{Witte}
\nocite{Borel}
\nocite{Leary}
\nocite{Knapp}
\nocite{Borel-Serre}
\nocite{Brown}
\nocite{Degrijse}
\nocite{Luck-Meintrup}
\nocite{Brady}
\nocite{Leary-Nucinkis}
\nocite{Martinez}
\nocite{Degrijse-Petrosyan}
\nocite{Djokovic}
\nocite{Gundogan}
\nocite{Gray-Wolf}
\nocite{Onishchik-Vinberg}
\nocite{Gundogane}
\nocite{Ji}
\nocite{Borell}
\nocite{Margulis}
\nocite{Degrijse-Souto}
\nocite{Aramayona-Martinez}
\nocite{Luck}
\nocite{Vogtmann}
\nocite{Pettet-Souto}
\nocite{Souto-Pettet}
\nocite{Gabai}
\nocite{Karass}
\nocite{Ash}

\hspace{0.5cm}

\textsc{IRMAR, Université de Rennes 1}

\textit{E-mail address}: \texttt{cyril.lacoste@univ-rennes1.fr}

\end{document}